\documentclass[reqno]{amsart}
\usepackage{amssymb,amsmath,enumerate,extarrows}
\usepackage[colorlinks=true,linkcolor=black,citecolor=black]{hyperref}
\usepackage[all]{xy}

\newcommand{\mm}{\mathfrak m}
\newcommand{\nn}{\mathfrak n}
\newcommand{\pp}{\mathfrak p}

\newcommand{\Z}{\mathbb{Z}}
\newcommand{\N}{\mathbb{N}}
\newcommand{\Fc}{\mathcal{F}}
\newcommand{\Gc}{\mathcal{G}}

\newcommand{\up}[1]{{{}^{#1}\!}}

\DeclareMathOperator{\chara}{char}

\DeclareMathOperator{\depth}{depth}
\DeclareMathOperator{\dstab}{dstab}
\DeclareMathOperator{\glind}{gl\,ld}
\DeclareMathOperator{\gr}{gr}

\DeclareMathOperator{\id}{id}

\DeclareMathOperator{\lcm}{lcm}
\DeclareMathOperator{\lind}{ld}
\DeclareMathOperator{\linp}{lin}
\DeclareMathOperator{\lstab}{lstab}
\DeclareMathOperator{\Min}{Min}
\DeclareMathOperator{\pnt}{\raise 0.5mm \hbox{\large\bf.}}
\DeclareMathOperator{\projdim}{pd}
\DeclareMathOperator{\pstab}{pstab}
\DeclareMathOperator{\reg}{reg}
\DeclareMathOperator{\rstab}{rstab}
\DeclareMathOperator{\supp}{supp}
\DeclareMathOperator{\Tor}{Tor}

\newcommand{\vphi}{\varphi}
\newcommand{\veps}{\varepsilon}

\newtheorem{thm}{\bf Theorem}[section]
\newtheorem{lem}[thm]{\bf Lemma}
\newtheorem{cor}[thm]{\bf Corollary}
\newtheorem{prop}[thm]{\bf Proposition}

\theoremstyle{definition}
\newtheorem{defn}[thm]{\bf Definition}

\theoremstyle{remark}
\newtheorem{rem}[thm]{Remark}
\newtheorem{ex}[thm]{Example}

\theoremstyle{remark}
\newtheorem*{rem*}{Remark}

\numberwithin{equation}{section}
\title[Powers of sums]{Powers of sums and their homological invariants}

\author{Hop D. Nguyen}
\address{Dipartimento di Matematica, Universit\`a di Genova, Via Dodecaneso 35, 16146 Genoa, Italy}
\address{Fakult\"at f\"ur Mathematik, Otto von Guericke Universit\"at Magdeburg, Universit\"atsplatz 2, 39106 Magdeburg, Germany}
\email{ngdhop@gmail.com}
\dedicatory{Dedicated to the memory of Diana Taylor \textup{(1941--2016)}}
\author{Thanh Vu}
\address{Department of Mathematics, University of Nebraska-Lincoln, Lincoln, NE 68588, USA}
\email{tvu@unl.edu}
\thanks{This work is partially supported by the NSF grant DMS-1103176.}


\begin{document}

\begin{abstract}
Let $R$ and $S$ be standard graded algebras over a field $k$, and $I \subseteq R$ and $J \subseteq S$ homogeneous ideals. Denote by $P$ the sum of the extensions of $I$ and $J$ to $R\otimes_k S$. We investigate several important homological invariants of powers of $P$ based on the information about $I$ and $J$, with focus on finding the exact formulas for these invariants. Our investigation exploits certain Tor vanishing property of natural inclusion maps between consecutive powers of $I$ and $J$. As a consequence, we provide fairly complete information about the depth and regularity of powers of $P$ given that $R$ and $S$ are polynomial rings and either $\chara k=0$ or $I$ and $J$ are generated by monomials.
\end{abstract}

\maketitle

\section{Introduction}
Let $R$ be a standard graded algebra over a field $k$, i.e. $R$ is a noetherian $\N$-graded ring with $R_0=k$ and $R$ is generated as an algebra over $k$ by elements of degree $1$. Let $\mm$ be the unique graded maximal ideal of $R$ and $I\subseteq \mm$ a homogeneous ideal of $R$. Studying the asymptotic behavior of the algebraic invariants associated to $I$ has been an important problem and has attracted much attention of commutative algebraists and algebraic geometers. A classical result states that the
Hilbert--Samuel function associated to any $\mm$-primary ideal is eventually a polynomial function. Another
well--known result due to Brodmann \cite{Br} establishes, for $s\gg 0$, that $\depth I^s$ is a
constant depending only on $I$.

In another direction, by results due to Cutkosky, Herzog, Kodiyalam, N.V. Trung and Wang \cite{CHT, K, TW}, the Castelnuovo--Mumford regularity $\reg I^s$ is a linear function of $s$ for $s\gg 0$. By definition, for a finitely generated graded $R$-module $M$, its Castelnuovo--Mumford regularity is $\reg M= \sup\{i+j: H^i_{\mm}(M)_j\neq 0\}$ where $H^i_{\mm}(M)$ denotes the $i$th local cohomology of $M$ with support in $\mm$. For more information on asymptotic properties of powers of ideals, see, e.g., \cite{BHH, Ch1, Ch, ChJR, EH, EU, HH1, NV, Ro2} and their references.

In this paper, we study several homological invariants of powers of sums of ideals. In detail, let $R$ and $S$ be standard graded $k$-algebras, $I$ and $J$ non-zero, proper homogeneous ideals of $R$ and $S$, respectively. Denote by $P$ the sum $I+J \subseteq T=R\otimes_k S$, where $I$ and $J$ are regarded as ideals of $T$. To avoid lengthy phrases later on, we call $P$ the {\it mixed sum} of $I$ and $J$. We consider the problem of characterizing (asymptotically) the homological invariants of powers of $P$, including the projective dimension, regularity, in terms of the data of $I$ and $J$. Besides the theory of asymptotic homological algebras of powers of ideals, another source of our motivation comes from recent work of H\`a, N.V. Trung and T.N. Trung \cite{HTT} on the depth and regularity of powers of $P$ in the case $R$ and $S$ are polynomial rings. While influenced by the last paper, the method of this paper is conceptually more transparent and yields more precise results under reasonable assumptions.

Our method makes substantial use of {\it Betti splittings}, first introduced by Francisco, H\`a and Van Tuyl \cite{FHV} in their study of free resolutions of monomial ideals. One of the main findings of this paper is that Betti splittings are ubiquitous and most relevant to study {\it all the powers} of mixed sums. The decomposition of $P$ as $I+J$ is probably the easiest example of a Betti splitting: letting $\beta_i(I)=\dim_k \Tor^R_i(k,I)$ be the $i$th Betti number of $I$, then the formula $\beta_i(P)=\beta_i(I)+\beta_i(J)+\beta_{i-1}(I\cap J)$ holds for any $i\ge 0$. These equations define Betti splittings in general and allow us to give a fairly complete understanding of the depth and regularity of the powers of $P$. We can prove, rather surprisingly, that the inequalities for $\depth T/P^s$ and $\reg T/P^s$ in \cite[Theorem 2.4]{HTT} are {\it both equalities} under certain conditions. 
\begin{thm}
\label{thm_depthreg}
Let $R$ and $S$ be polynomial rings over $k$. Suppose that either $\chara k=0$ or $I$ and $J$ are monomial ideals. Then for any $s\ge 1$, there are equalities
\begin{enumerate}[\quad \rm(i)]
 \item $\depth T/P^s=$
 $$\min_{i\in [1,s-1], j\in [1,s]}\left\{\depth R/I^{s-i}+\depth S/J^i+1,\depth R/I^{s-j+1}+\depth S/J^j\right\},$$
\item $\reg T/P^s=$
$$\max_{i\in [1,s-1],j\in [1,s]}\left\{\reg R/I^{s-i}+\reg S/J^i+1,\reg R/I^{s-j+1}+\reg S/J^j\right\}.$$
\end{enumerate}
\end{thm}
Theorem \ref{thm_depthreg} is surprising since normally we only expect to get hold of information about large enough powers of an ideal. Prior to our result, it was suspected that there are no general formulae for the depth and regularity of $T/P^s$ even when $R$ and $S$ are polynomial rings (see \cite[Page 820]{HTT}). Thanks to Theorem \ref{thm_depthreg}, we can also provide an upper bound for the stability index of depth of powers of $P$, and study when $P$ has a constant depth function (Section \ref{subsect_applications}). Our results generalize or strengthen previous work of H\`a--Trung--Trung \cite{HTT}, Herzog--Vladoiu \cite{HV}. We expect Theorem \ref{thm_depthreg} to hold in positive characteristics as well (but by Example \ref{ex_nonpolynomialbase} it does not hold if one of the base rings is not regular).

We do not restrict our considerations of mixed sums to the setting of polynomial base rings. To study Betti splittings of higher powers of $P$, we introduce in Section \ref{sect_idealsofsmalltype} the following notion.
\begin{defn}
We say $I$ is an ideal {\it of small type} if for all $s\ge 1$ and all $i\ge 0$, the map $\Tor^R_i(k,I^s)\to \Tor^R_i(k,I^{s-1})$ induced by the natural inclusion $I^s\to I^{s-1}$ is zero.
\end{defn}
One of the main results of Section \ref{sect_idealsofsmalltype} is Theorem \ref{thm_sumsofsmalldoublysmall}: if $I$ and $J$ are of small type then the powers of $P$ admit natural Betti splittings. While simple, Theorem \ref{thm_sumsofsmalldoublysmall} is enough to produce the formulas of Theorem \ref{thm_depthreg} since if $R$ and $S$ are polynomial rings then $I$ and $J$ are of small type if either $\chara k=0$ (Ahangari Maleki \cite[Proposition 3.5]{A}) or $I$ and $J$ are monomial ideals (Theorem \ref{thm_monomialidealsareofsmalltype}).

Another goal of this paper is to study the linearity defect of powers of mixed sums. The linearity defect, introduced by Herzog and Iyengar \cite{HIy} measures the failure of minimal free resolutions to be linear (see Section \ref{sect_background} for more details). From the theory of componentwise linear ideals initiated by Herzog and Hibi \cite{HH}, linearity defect is a natural invariant: over a polynomial ring, componentwise linear ideals are exactly ideals with linearity defect zero (see \cite[Section 3.2]{Ro}). The linearity defect, as can be expected, behaves well asymptotically. Denote by $\lind_R M$ the linearity defect of a finitely generated graded $R$-module $M$. We prove in \cite[Theorem 1.1]{NgV1b} that if $R$ is a polynomial ring, then the sequence $(\lind_R I^s)_{s\ge 1}$ is eventually constant. To understand the meaning of the limit $\lim_{s\to \infty} \lind_R I^s$,  one should naturally look for some computations of the asymptotic linearity defect. This is the original motivation of this paper which led to the discovery of Betti splittings of powers of mixed sums. The main results of Section \ref{sect_lindofpowers} yield the following computational statement.
\begin{thm}[Corollary \ref{cor_asymptote_specialcases}]
\label{thm_ldofpowers_specialcases}
Let $(R,\mm)$ and $(S,\nn)$ be polynomial rings over $k$, $(0) \neq I \subseteq \mm^2$ and $(0) \neq J \subseteq \nn^2$ homogeneous ideals. Suppose that one of the following conditions holds:
\begin{enumerate}[\quad \rm(i)]
\item $\chara k=0$ and $I$ and $J$ are non-trivial powers of some homogeneous ideals of $R$ and $S$, resp.
\item $I$ and $J$ are non-trivial powers of some monomial ideals of $R$ and $S$, resp.
\item All the powers of $I$ and $J$ are componentwise linear.
\end{enumerate}
Then for all $s\gg 0$, we have an equality
\[
\lind_T P^s=\max\left\{\lim_{i\to \infty}\lind_R I^i+\max_{j \ge 1}\lind_S J^j+1 , \max_{i \ge 1}\lind_R I^i+\lim_{j\to \infty}\lind_S J^j+1\right\}.
\]
\end{thm}
We believe that Theorem \ref{thm_ldofpowers_specialcases} is still true without the extra assumptions (i), (ii) and (iii).

The paper is organized as follows. Section \ref{sect_background} is devoted to the necessary background and useful facts. In Section \ref{sect_Torvanishing_andinvariants}, we introduce Betti splittings and the closely related notion of Tor-vanishing morphisms between modules. Our study of powers of mixed sums rests on the fact that several homological invariants behave well with respect to  Betti splittings and Tor-vanishing morphisms. In Section \ref{sect_idealsofsmalltype}, we study ideals of small type and the subclass of ideals of doubly small type. Theorem \ref{thm_ldofpowers_specialcases} is available thanks to the formula for $\lind_T P^s$ when $I$ and $J$ are of doubly small type (Theorem \ref{thm_boundld_mixedsum}(ii)). The first main result of Section \ref{sect_idealsofsmalltype} is Theorem \ref{thm_monomialidealsareofsmalltype} establishing that any proper monomial ideal of a polynomial ring is of small type. The second main result of this section is Theorem \ref{thm_sumsofsmalldoublysmall} which, among other things, establishes Betti splittings for powers of mixed sums whose summands are of small type. We study in Sections \ref{sect_depthreg} and \ref{sect_lindofpowers} the homological invariants, including projective dimension, regularity and linearity defect, of powers of mixed sums. The main focus in both sections is to provide exact formulas for these invariants. We also compute the asymptotic values of these invariants whenever possible. Some applications to the theory of ideals with constant depth functions and the index of depth stability are given at the end of Section \ref{sect_depthreg}.

{\small 
\begin{rem*}
Some materials of this paper originally belong to a preprint titled ```Linearity defects of powers are eventually constant" \cite{NgV1a}. After further thoughts, we split the last preprint into three parts. The first part establishes the asymptotic constancy of the linearity defect of powers \cite{NgV1b}. This paper is the second part, which contains Section 4 of \cite{NgV1a} but goes significantly beyond that. A third part which contains the results in the last section of \cite{NgV1a}, among other things, is the preprint \cite{NgV1c}.
\end{rem*} 
}
%--------------------------------------------------------
%--------------------------------------------------------
%--------------------------------------------------------
%--------------------------------------------------------

\section{Background}
\label{sect_background}
We begin with some basic notions and facts that will be used later. Standard knowledge of commutative algebra may be found in \cite{BH}, \cite{Eis}. For the theory of free resolutions, we refer to \cite{Avr} and \cite{P}.
\subsection{Linearity defect}
Let $(R,\mm,k)$ be a noetherian ring which is one of the following:
\begin{enumerate}
\item a local ring with the maximal ideal $\mm$ and the residue field $k=R/\mm$,
\item a standard graded algebra over a field $k$ with the graded maximal ideal $\mm$.
\end{enumerate}
We usually omit $k$ and write simply $(R,\mm)$. Let $M$ be a finitely generated $R$-module. If $R$ is a graded algebra, then the $R$-modules that we will study are assumed to be graded, and various structures concerning them, e.g.~ their minimal free resolutions are also taken in the category of graded $R$-modules and degree preserving homomorphisms. 

Let $F$ be the minimal free resolution of $M$:
\[
F: \quad	\cdots \longrightarrow F_i \longrightarrow F_{i-1} \longrightarrow \cdots \longrightarrow F_1 \longrightarrow F_0 \longrightarrow 0.
\]
For each $i\ge 0$, the minimality of $F$ gives rise to the following subcomplex of $F$:
\[
\Fc^iF:  \quad	\cdots \longrightarrow F_{i+1} \longrightarrow  F_i \longrightarrow  \mm F_{i-1}  \longrightarrow  \cdots \longrightarrow \mm^{i-j} F_j  \longrightarrow  \cdots.
\]
Following Herzog and Iyengar \cite{HIy}, we define the so-called {\em linear part} of $F$ by the formula
$$
\linp^R F:=\bigoplus_{i=0}^{\infty} \frac{\Fc^iF}{\Fc^{i+1}F}.
$$
Prior to the work of Herzog and Iyengar, the linear part was also studied by Eisenbud et. al \cite{EFS} in the graded setting. Observe that $(\linp^R F)_i=(\gr_{\mm}F_i)(-i)$ for every $i\ge 0$, where $\gr_{\mm} M=\bigoplus_{i\ge 0} \mm^iM/(\mm^{i+1}M)$. Note that $\linp^R F$ is a complex of graded modules over $\gr_{\mm}R$. The construction of $\linp^R F$ has a simple interpretation in the graded case. For each $i\ge 1$, apply the following rule to all entries in the matrix representing the map $F_i\to F_{i-1}$: keep it if it is a linear form, and replace it by $0$ otherwise. Then the resulting complex is $\linp^R F$.

Following \cite{HIy}, the linearity defect of $M$ is
\[
\lind_R M:=\sup\{i: H_i(\linp^R F)\neq 0\}.
\]
If $M\cong 0$, we set $\lind_R M=0$. 

Most results on the linearity defect of the present paper are built upon the following theorem.
\begin{thm}[\c{S}ega, {\cite[Theorem 2.2]{Se}}]
\label{thm_Sega}
Let $d\ge 0$ be an integer. The following statements are equivalent:
\begin{enumerate}[\quad \rm(i)]
\item $\lind_R M\le d$;
\item The natural morphism $\Tor^R_i(R/\mm^{q+1}, M)\longrightarrow \Tor^R_i(R/\mm^q, M)$ is zero for every $i>d$ and every $q\ge 0$.
\end{enumerate}
\end{thm}

If a module has linearity defect $0$, then as in \cite{HIy}, we also say that it is a {\em Koszul module}. Note that by our convention, the trivial module $(0)$ is a Koszul module. The ring $R$ is called {\it Koszul} if $\lind_R k=0$. In the graded case, $R$ is said to be a {\it Koszul algebra} if $\reg_R k=0$. These two notions are compatible in the graded case by \cite[Proposition 1.13]{HIy}. Moreover by \cite[Remark 1.10]{HIy}, a noetherian local ring $(R,\mm)$ is Koszul if and only if $\gr_{\mm}R$ is a Koszul algebra. 

\subsection{Regularity}
\label{subsect_regularity}
Let $R$ be a standard graded $k$-algebra, $M$ a finitely generated graded $R$-module. The number $\beta_{i,j}(M)=\dim_k \Tor^R_i(k,M)_j$ is called the $(i,j)$ graded Betti number of $M$ and $\beta_i(M)=\dim_k \Tor^R_i(k,M)$ its $i$th Betti number. We define the regularity of $M$ over $R$ as follows
\[
\reg_R M=\sup\{j-i: \beta_{i,j}(M) \neq  0\}.
\] 
If $\lind_R M<\infty$ then by \cite[Proposition 3.5]{AR}, there is an equality 
\[
\reg_R M= \sup_{0\le i\le \lind_R M}\{j-i: \Tor^R_i(k,M)_j\neq 0\}.
\] 
In particular, if $M$ is a Koszul module then $\reg_R M$ equals the maximal degree of a minimal homogeneous generator of $M$.

Recall that $M$ is said to have a $d$-linear resolution (where $d\in \Z$), if for all $i\in \Z$ and all $j\neq d$, it holds that $\Tor^R_i(k,M)_{i+j}=0$. If $M$ has a $d$-linear resolution, then necessarily $M$ is generated in degree $d$ and $\reg_R M=d$.

Following Herzog and Hibi \cite{HH}, $M$ is said to be {\it componentwise linear} if for each $d$, the submodule generated by homogeneous elements of degree $d$ in $M$ has a $d$-linear resolution. If $R$ is a Koszul algebra, then we have the following implications, with the first one being strict in general:
\[
\text{$M$ has a linear resolution} \Rightarrow  \text{$M$ is componentwise linear}  \Leftrightarrow \text{$M$ is Koszul}.
\]
The equivalence is due to R\"omer \cite[Theorem 3.2.8]{Ro}. Throughout, we use the term ``Koszul module" instead of ``componentwise linear module" to streamline the exposition.

The {\it Castelnuovo--Mumford regularity} of $M$ measures the degree of vanishing of its local cohomology modules: $
\reg M=\sup\{i+j: H^i_{\mm}(M)_j\neq 0\}$, where $H^i_{\mm}(M)$ denotes the $i$th local cohomology module of $M$.

\subsection{Elementary lemmas}
The identities in the next lemma extend \cite[Lemma 1.1]{HT} and \cite[Proposition 3.2]{HTT}. They are the basis of most discussions about mixed sums in this paper. Note that the proof of \cite[Proposition 3.2]{HTT} is defective since the authors assume following false claim: Let  $M_1,\ldots,M_p,N$ are submodules of an ambient module $M_0$ over a ring $R$ (where $p\ge 1$) such that $M_i\cap M_j\subseteq N$ for all $i\neq j$. Denote $M=M_1+\cdots+M_p$. Then the module $(M+N)/N$ admits a direct sum decomposition $(M+N)/N=\bigoplus_{i=1}^p (M_i+N)/N$. To see that the claim is false, consider the following example. Let $R$ be the field $k$, $M_0=kx\oplus ky$ be a $k$-vector space of rank 2, $M_1=kx$, $M_2=ky$ and $N=k(x+y)$. Then $M_1\cap M_2=0 \subseteq N$. On the other hand $M/N=(M_1+N)/N=(M_2+N)/N$ is the $k$-vector space with basis $x+N$.
\begin{lem}
\label{lem_intersect_quotient}
Let $R,S$ be affine $k$-algebras. Let $I,J$ be ideals of $R,S$, resp., and $P=I+J \subseteq T=R\otimes_k S$. Then for all $p,q,r,s\ge 0$, there are identities:
\begin{align}
I^{p+r}J^q\cap I^rJ^{s+q}&=I^{p+r}P^q\cap I^rJ^{s+q}=I^{p+r}J^{s+q}, \label{eq_intersection}\\
\frac{I^rP^s}{I^rP^{s+1}}&\cong \bigoplus_{i=0}^s \left(\frac{I^{r+i}}{I^{r+i+1}}\otimes_k \frac{J^{s-i}}{J^{s-i+1}}\right). \label{eq_decompositionofquotient}
\end{align}
\end{lem} 
\begin{proof}
\eqref{eq_intersection} First we treat the case $p=s=1,r=q=0$. Since $k$ is a field, we have the following identities
\[
I\cap J =(I \otimes_k S)\cap (R\otimes_k J)=(I \cap R)\otimes_k (S\cap J)=I \otimes_k J=IJ.
\]
Now for arbitrary $0\le r\le p, 0\le q\le s$, we have
\[
I^{p+r}J^{s+q} \subseteq I^{p+r}J^q\cap I^rJ^{s+q} \subseteq I^{p+r}P^q\cap I^rJ^{s+q} \subseteq I^{p+r}\cap J^{s+q}=I^{p+r}J^{s+q}.
\]
The last equality in the display follows from the case treated above. Hence all the inclusions are in fact equalities.

\eqref{eq_decompositionofquotient} We have $I^rP^s=\sum_{i=0}^{s}I^{r+i}J^{s-i}$. We claim that for $0\le i \le s$, the following inclusion holds
\begin{equation}
\label{inclusion_Ir+iJs-i}
I^{r+i}J^{s-i} \bigcap \left(\sum\limits_{0\le j\le s, j\neq i}I^{r+j}J^{s-j}+I^rP^{s+1}\right) \subseteq I^rP^{s+1}.
\end{equation}

Note that $P^{s+1} \subseteq I^{i+1}+J^{s-i+1}$, so $I^rP^{s+1} \subseteq I^{r+i+1}+J^{s-i+1}$. Moreover $J^{s-j}\subseteq J^{s-i+1}$, $I^{r+t} \subseteq I^{r+i+1}$ if $j<i<t$. Hence
\begin{align}
&I^{r+i}J^{s-i} \cap \left(\sum\limits_{0\le j\le s, j\neq i}I^{r+j}J^{s-j}+I^rP^{s+1}\right) \nonumber \\
 &=I^{r+i}J^{s-i} \cap \left(\sum_{j<i}I^{r+j}J^{s-j}+\sum_{j>i}I^{r+j}J^{s-j}+I^rP^{s+1}\right) \nonumber\\
& \subseteq I^{r+i}J^{s-i} \cap \left(I^{r+i+1}+J^{s-i+1}\right). \label{ineq_intersectIr+iJs-i}
\end{align}
Generally for all ideals $I_1,I_2,I_3$ in a common ring, we have
\[
I_1\cap (I_2+I_3) \subseteq \left(I_2\cap (I_1+I_3)\right)+\left(I_3\cap (I_1+I_2)\right).
\] 
Take $I_1=I^{r+i}J^{s-i}, I_2=I^{r+i+1}, I_3=J^{s-i+1}$ in $T$. Then as $I_1+I_3 \subseteq J^{s-i}, I_1+I_2\subseteq I^{r+i}$, the last display implies the inclusion in the following chain
\begin{align}
\label{eq_inclusion_intersect}
I^{r+i}J^{s-i} \cap (I^{r+i+1}+J^{s-i+1}) &\subseteq I^{r+i+1}\cap J^{s-i}+J^{s-i+1} \cap I^{r+i} \nonumber \\
&=I^{r+i+1}J^{s-i}+I^{r+i}J^{s-i+1}.
\end{align}
The equality holds because of \eqref{eq_intersection}. Note that the last chain also yields
\begin{equation}
\label{eq_intersectwithIrPs+1}
I^{r+i}J^{s-i}\cap I^rP^{s+1}=I^{r+i+1}J^{s-i}+I^{r+i}J^{s-i+1}.
\end{equation}
Indeed, the right-hand side is trivially contained in the left-hand one. For the reverse inclusion, we only need to recall that $I^rP^{s+1} \subseteq I^{r+i+1}+J^{s-i+1}$. The chain \eqref{eq_inclusion_intersect} takes care of the rest.

The combination of \eqref{ineq_intersectIr+iJs-i} with \eqref{eq_inclusion_intersect} and \eqref{eq_intersectwithIrPs+1} yields the desired inclusion \eqref{inclusion_Ir+iJs-i}. So we have the first equality in the following chain
\begin{align*}
\frac{I^rP^s}{I^rP^{s+1}}&=\bigoplus_{i=0}^s\frac{I^{r+i}J^{s-i}+I^rP^{s+1}}{I^rP^{s+1}} \cong \bigoplus_{i=0}^s\frac{I^{r+i}J^{s-i}}{I^{r+i}J^{s-i} \cap I^rP^{s+1}}\\
 &=\bigoplus_{i=0}^s\frac{I^{r+i}J^{s-i}}{I^{r+i+1}J^{s-i}+I^{r+i}J^{s-i+1}} \cong \bigoplus_{i=0}^s \left(\frac{I^{r+i}}{I^{r+i+1}}\otimes_k \frac{J^{s-i}}{J^{s-i+1}}\right).
\end{align*}
The second equality is \eqref{eq_intersectwithIrPs+1}, the last isomorphism is standard. This finishes the proof.
\end{proof}
\begin{lem}
\label{lem_tensor}
Let $R,S$ be standard graded $k$-algebras, and $M, N$ be non-zero finitely generated graded modules over $R, S$, resp. Then denoting $T=R\otimes_k S$, there are equalities
\begin{align*}
\projdim_T (M\otimes_k N)&=\projdim_R M+\projdim_S N,\\
\reg_T (M\otimes_k N)&=\reg_R M+ \reg_S N,\\
\lind_T (M\otimes_k N)&=\lind_R M+\lind_S N.
\end{align*}
\end{lem}
\begin{proof}
Let $F, G$ be the minimal graded free resolutions of $M,N$ over $R,S$, resp. Then $F \otimes_k G$ is a minimal graded free resolution of $M \otimes_k N$ over $T$. A simple accounting then yields the first two equalities. 

We immediately have $\linp^T (F \otimes_k G)=\linp^R F \otimes_k \linp^S G$. Hence using the K\"unneth's formula, we get $\lind_T (M\otimes_k N)=\lind_R M+\lind_S N$.
\end{proof}
Recall that a map of noetherian local rings $\theta:  (R,\mm) \longrightarrow (S,\nn)$ is called an {\it algebra retract} if there exists a local homomorphism $\phi: S\to R$ such that the composition $\phi\circ \theta$ is the identity of $R$. In such a situation, we call the map $\phi$ the {\it retraction map} of $\theta$. 
\begin{lem}
\label{lem_retract}
Let $\theta: (R,\mm)\to (S,\nn)$ be an algebra retract of noetherian local rings with the retract map $\phi: S\to R$. Let $I\subseteq \mm$ be an ideal of $R$. Let $J\subseteq \nn$ be an ideal containing $\theta(I)S$ such that $\phi(J)R=I$. Then there are inequalities
\begin{align*}
\lind_R (R/I) &\le \lind_S (S/J),\\
\lind_R I &\le \lind_S J. 
\end{align*}
\end{lem}
\begin{proof}
The hypothesis implies that there is an induced algebra retract $R/I \xrightarrow{\theta} S/J$. For each $i\ge 0, q\ge 0$, there is a commutative diagram of $R$-modules
\[
\xymatrixcolsep{6mm}
\xymatrixrowsep{6mm}
\xymatrix{
\Tor^R_i(R/\mm^{q+1},R/I) \ar[rr]^{\iota^{q+1}} \ar[d]^{\mu^q_R} && \Tor^S_i(S/\nn^{q+1},S/J) \ar[rr] \ar[d]^{\mu^q_S} && \Tor^R_i(R/\mm^{q+1},R/I) \ar[d] \\
\Tor^R_i(R/\mm^q,R/I) \ar[rr]^{\iota^q}  && \Tor^S_i(S/\nn^q,S/J) \ar[rr]        && \Tor^R_i(R/\mm^q,R/I)                              
}
\]
By functoriality, the composition of the horizontal maps on the second row is the identity map of $\Tor^R_i(R/\mm^q,R/I)$. From this, we deduce that $\iota^q$ is injective. 
	
Take $i>\lind_S (S/J)$, then from Theorem \ref{thm_Sega}, $\mu^q_S$ is trivial for all $q\ge 0$. Since $\iota^q$ is injective, we also have $\mu^q_R$ is trivial for all $q\ge 0$. Using Theorem \ref{thm_Sega}, this implies that $\lind_R (R/I)\le \lind_R (S/J)$. The remaining inequality is immediate.
\end{proof}
%-----------------------------------------------------
%-----------------------------------------------------
%-----------------------------------------------------
%-----------------------------------------------------
%-----------------------------------------------------

\section{Maps of Tor and algebraic invariants}
\label{sect_Torvanishing_andinvariants}
This section makes the preparation for Section \ref{sect_idealsofsmalltype}. We will introduce the notion of Tor-vanishing and doubly Tor-vanishing homomorphisms. Morphisms of both type are well-suited to study the linearity defect but the latter yields more precise information. We also recall some results from \cite{Ng1}, \cite{NgV2}, which will be used frequently later. 
\subsection{Tor-vanishing morphisms}
Let $(R,\mm)$ be a noetherian local ring. Let $\phi: M\to P$ be a morphism of finitely generated $R$-modules. We say that $\phi$ is {\it Tor-vanishing} if for all $i\ge 0$, we have $\Tor^R_i(k,\phi)=0$. We say that $\phi$ is {\it doubly Tor-vanishing} if there exist (possibly non-minimal) free resolutions $F$ and $G$ of $M$ and $P$, resp., and a lifting $\vphi: F\to G$ of $\phi$ such that $\vphi(F)\subseteq \mm^2G$.

We use the same terminology for the setting of standard graded algebras over a field and graded modules. Clearly, if $\phi$ is doubly Tor-vanishing then it is Tor-vanishing, but not vice versa: For example, take $R=k[[x]]$ and let $\phi$ be the map $R\xrightarrow{\cdot x}R$. 
\begin{rem}
\label{rem_equivalenceTor-vanishing}
Let $F,G$ be the minimal free resolutions of $M, P$, resp. 

(1) The following statements are equivalent:
\begin{enumerate}
\item $\phi$ is Tor-vanishing;
\item There exists a lifting $\vphi: F\to G$ of $\phi$ such that $\vphi(F)\subseteq \mm G$;
\item For any lifting $\vphi: F\to G$ of $\phi$, we have $\vphi(F) \subseteq \mm G$.
\end{enumerate}

(2) The following statements are equivalent:
\begin{enumerate}
\item $\phi$ is doubly Tor-vanishing;
\item There exists a lifting $\vphi: F\to G$ of $\phi$ such that $\vphi(F)\subseteq \mm^2G$;
\item For any lifting $\vphi: F\to G$ of $\phi$, we have $\vphi(F) \subseteq \mm^2G$.
\end{enumerate}
\end{rem}

Compositing Tor-vanishing morphisms produce doubly Tor-vanishing ones.
\begin{lem}
\label{lem_composition_Tor-vanishing}
Let $M\xrightarrow{\phi}P$ and $P\xrightarrow{\psi}Q$ be Tor-vanishing morphisms. Then $\psi  \circ \phi$ is doubly Tor-vanishing.
\end{lem}
\begin{proof}
Let $F, G, H$ be the minimal free resolutions of $M, P, Q$, resp. Since $\phi$ is Tor-vanishing, it admits a lifting $\vphi: F\to G$ such that $\vphi(F) \subseteq \mm G$. Similarly, there exists a lifting $\Psi: G\to H$ of $\psi$ such that $\Psi(G)\subseteq \mm H$. The composition $\Psi\circ \vphi: F\to H$ yields a lifting of $\psi\circ \phi$ such that $(\Psi \circ \vphi)(F) \subseteq \mm^2H$.
\end{proof}
The following result illustrates the utility of (doubly) Tor-vanishing morphisms.
\begin{lem}
\label{lem_Tor-vanishingandinvariants}
Let $0\to M \xrightarrow{\phi} P \to N \to 0$ be an exact sequence of non-zero finitely generated $R$-modules.
\begin{enumerate}[\quad \rm(i)]
\item Assume that $\phi$ is Tor-vanishing. Then 
$$
\projdim_R N=\max\{\projdim_R P, \projdim_R M+1\}.
$$
If moreover $R$ is a standard graded algebra and $M, P, N$ are graded modules, then $\reg_R N=\max\{\reg_R P, \reg_R M-1\}.$
\item Assume that $\phi$ is doubly Tor-vanishing. Then $\lind_R N=\max\{\lind_R P, \lind_R M+1\}.$
\end{enumerate}
\end{lem}
\begin{proof}
(i) The long exact sequence of Tor yields $\beta_i(N)=\beta_i(P)+\beta_{i-1}(M)$ for all $i\ge 0$. Hence the equality for  the projective dimension follows. Similar arguments for the regularity. (ii) follows by using Remark \ref{rem_equivalenceTor-vanishing}(2) and \cite[Lemma 2.2]{NgV1b}.
\end{proof} 

\subsection{Betti splittings}
Betti splittings were first introduced by Francisco, H\`a, and Van Tuyl \cite{FHV} for monomial ideals. A motivation for this notion comes from work of Eliahou and Kervaire \cite{EK}. Let $(R,\mm)$ be a noetherian local ring (or a standard graded $k$-algebra) and $P,I,J\neq (0)$ be proper (homogeneous) ideals of $R$ such that $P=I+J$. 
\begin{defn}
The decomposition of $P$ as $I+J$ is called a {\it Betti splitting} if for all $i\ge 0$, the following equality of Betti numbers holds: $\beta_i(P)=\beta_i(I)+\beta_i(J)+\beta_{i-1}(I\cap J)$.
\end{defn}
The lemma below is a straightforward generalization of \cite[Proposition 2.1]{FHV} and admits the same proof. We will frequently use the characterization (ii) of Betti splittings in the next sections.
\begin{lem}
\label{lem_criterion_Bettisplit}
The following are equivalent:
\begin{enumerate}[\quad \rm(i)]
\item The decomposition $P=I+J$ is a Betti splitting;
\item The morphisms $I\cap J \to I$ and $I\cap J \to J$ are Tor-vanishing;
\item The mapping cone construction for the map $I\cap J \to I\oplus J$ yields a minimal free resolution of $P$.
\end{enumerate}
\end{lem}
Most results of this paper are motivated by the next simple observation, presented in \cite[Example 4.7]{NgV2}.
\begin{ex}
\label{ex_Bettisplittings}
Let $R, S$ be standard graded $k$-algebras and $I, J$ be non-zero, proper homogeneous ideals of $R, S$, resp. Let $P=I+J \subseteq T=R\otimes_k S$, then the decomposition $P=I+J$ is a Betti splitting.
\end{ex}
The following result signifies the utility of Betti splittings. The case when $R$ is a polynomial ring is well-known \cite[Corollary 2.2]{FHV}.
\begin{lem}
\label{lem_depthreg_Bettisplit}
Let $R$ be a standard graded $k$-algebra and $P$ a homogeneous ideal with a Betti splitting $P=I+J$. Then there are equalities
\begin{align*}
\projdim_R P&=\max\{\projdim_R I, \projdim_R J, \projdim_R (I\cap J)+1\},\\
\reg_R P &= \max\{\reg_R I, \reg_R J, \reg_R (I\cap J)-1\}.
\end{align*}
\end{lem}
\begin{proof}
Looking at the long exact sequence of $\Tor^R(k,-)$, we see that the second equality is always true and the first one is true if $I\cap J\neq (0)$. If $I\cap J=(0)$, then $P=I\oplus J$ so  $\projdim_R P=\max\{\projdim_R I, \projdim_R J\}$. If $\max\{\projdim_R I, \projdim_R J\}\ge 1$ then the first equality is again true. If $\projdim_R I=\projdim_R J=0$ then $I$ and $J$ are free $R$-modules. Hence there are non-zero divisors $x\in I, y\in J$. But then $0\neq xy\in I\cap J$, a contradiction. This finishes the proof of the first equality. 
\end{proof}
\subsection{Exact sequence estimates}
Let $(R,\mm)$ be a noetherian local ring and $0\longrightarrow M\xlongrightarrow{\phi} P \xlongrightarrow{\lambda} N\longrightarrow 0$ be an exact sequence of finitely generated $R$-modules. From \c{S}ega's Theorem \ref{thm_Sega}, it follows that the vanishing of $\Tor^R_{\pnt}(k,\phi), \Tor^R_{\pnt}(k,\lambda)$ and the connecting map $\Tor^R_{\pnt+1}(k,N)\to \Tor^R_{\pnt}(k,M)$ are useful for comparing the numbers $\lind_R M, \lind_R P$ and $\lind_R N$. For example, we have
\begin{prop}[{\cite[Proposition 4.3]{NgV2}}]
\label{prop_zeroinducedmapofTor}
With the notations as above, assume that $\phi$ is Tor-vanishing. Then there are inequalities
\begin{align*}
\lind_R N &\le \max\{\lind_R P, \lind_R M+1\},\\
\lind_R P &\le \max\{\lind_R M, \lind_R N\},\\
\lind_R M &\le \max\{\lind_R P, \lind_R N-1\}.
\end{align*}
\end{prop}
Using Lemma \ref{lem_criterion_Bettisplit} and Proposition \ref{prop_zeroinducedmapofTor} for the short exact sequence $0\longrightarrow I\cap J \longrightarrow I\oplus J \longrightarrow P \longrightarrow 0$, we get
\begin{prop}[{\cite[Theorem 4.9]{NgV2}}]
\label{prop_Bettisplittings}
Let $P=I+J$ be a Betti splitting of non-zero proper ideals of $R$. Then there are inequalities
\begin{align*}
\lind_R P &\le \max\{\lind_R I, \lind_R J, \lind_R (I\cap J)+1\},\\
\max\{\lind_R I, \lind_R J\} &\le \max\{\lind_R (I\cap J), \lind_R P\},\\
\lind_R (I\cap J) &\le \max\{\lind_R I, \lind_R J, \lind_R P-1\}.
\end{align*}
\end{prop}
The following result will be invoked several times. Note that, compared with the original statements in \cite{Ng1}, below we additionally allow trivial modules. No contradiction arises in doing so because of the convention that the trivial module is Koszul.
\begin{thm}[Nguyen {\cite[Theorem 3.5 and its proof]{Ng1}}]
\label{thm_m-smallext}
Let $0\longrightarrow M\xlongrightarrow{\phi} P \longrightarrow N \longrightarrow 0$ be a short exact sequence of finitely generated $R$-modules such that:
\begin{enumerate}[\quad\rm(i)]
\item $P$ is a Koszul module,
\item $M\subseteq \mm P$.
\end{enumerate}
Then $\phi$ is Tor-vanishing, and $\lind_R N-1 \le \lind_R M \le \max\{0,\lind_R N-1\}$.  Furthermore, $\lind_R N=0$ if and only if $\lind_R M=0$ and $M\cap \mm^{s+1}P=\mm^sM$ for all $s\ge 0$.
\end{thm}

\subsection{Invariants of mixed sums}
Let $(R,\mm)$ and $(S,\nn)$ be standard graded $k$-algebras, $I\subseteq \mm$ and $J \subseteq \nn$ be non-zero homogeneous ideals of $R$ and $S$, resp. In the sequel, modules over $R$ or $S$ are identified with their extensions to $T$ (via the obvious faithfully flat maps). For simplicity, we will call $P=I+J$ the {\em mixed sum} of $I$ and $J$. More generally, if $R_1,\ldots,R_c$ are standard graded $k$-algebras (where $c\ge 2$) and $I_i\subseteq R_i$ is a homogeneous ideal for $1\le i\le c$, we call $I_1+\cdots+I_c \subseteq R_1\otimes_k \cdots \otimes_k R_c$ the mixed sum of $I_1,\ldots,I_c$.

Part (i) of the following result is folkloric.
\begin{prop}
\label{prop_invariants:mixedsum}
The following statements hold. 
\begin{enumerate}[\quad \rm(i)]
\item There are equalities
\begin{align*}
\projdim_T P &=\projdim_R I+\projdim_S J+1,\\
\reg_T P     &=\reg_R I+\reg_S J-1.
\end{align*}
\item If $R/I$ is a Koszul $R$-module then $\lind_T P=\lind_S J$. If $\lind_R (R/I)$ and $\lind_S (S/J)$ are $\ge 1$ then $\lind_T P=\lind_R I+\lind_S J+1.$
\end{enumerate}
\end{prop}
\begin{proof}
(i) We prove the equality for projective dimension; the same argument works for regularity.

By Lemma \ref{lem_tensor}, we get the second equality in the display
\begin{align*}
\projdim_T P&=\projdim_T(T/P)-1=\projdim_R(R/I)+\projdim_S(S/J)-1\\
            &=\projdim_R I+\projdim_S J+1.
\end{align*}

(ii) Again by Lemma \ref{lem_tensor}, we have $\lind_T (T/P)=\lind_R (R/I) +\lind_S (S/J)$. This implies the first part of (ii). For the second part, note that $\lind_R (R/I)\ge 1$, hence $\lind_R I=\lind_R (R/I)-1$. Similar equalities hold for $J$ and $P$, and the conclusion follows.
\end{proof}
\begin{rem}
If $R$ is a Koszul algebra, then by Subsection \ref{subsect_regularity}, $R/I$ is a Koszul module $\Longleftrightarrow$ $R/I$ has a linear resolution over $R$ $\Longleftrightarrow$ $I$ has a $1$-linear resolution over $R$.

In any case, if $I$ is not generated by linear forms then $\lind_R(R/I) \ge 1$. Hence the hypotheses of Proposition \ref{prop_invariants:mixedsum}(ii) are satisfied if $I\subseteq \mm^2, J\subseteq \nn^2$. Later on, we will have different kinds of results about $\lind_T P^s, s\ge 2$ according to whether $I$ is generated by linear forms or $I \subseteq \mm^2$ (compare for example Corollary \ref{cor_asymptote_specialcases} and Theorem \ref{thm_asymptoticlind_linear}).
\end{rem}
%-------------------------------------------------------
%-------------------------------------------------------
%-------------------------------------------------------
%-------------------------------------------------------
%-------------------------------------------------------
%-------------------------------------------------------

\section{Ideals of (doubly) small type}
\label{sect_idealsofsmalltype}
Let $(R,\mm)$ be a noetherian local ring (or a standard graded $k$-algebra), and $I$ a proper (homogeneous) ideal. We say that $I$ is {\it of small type} ({\it of doubly small type}) if for all $r\ge 1$, the natural map $I^r\to I^{r-1}$ is Tor-vanishing (doubly Tor-vanishing, resp.). Ideals which are {\it not} of small type abound: By \c{S}ega's \cite[Proposition 7.5]{Se1}, if $R$ is a non-Koszul local ring, then $\mm$ is not of small type (and vice versa). Moreover, even over some Koszul complete intersections of codimension $2$ ideals not of small type exist, e.g. consider the ideal $(x+y)$ of $k[x,y,z]/(x^2,yz)$ in Example \ref{ex_notofsmalltype_ld}.

In this section, we provide diverse classes of ideals of small or doubly small type. An important result in the current section is Theorem \ref{thm_sumsofsmalldoublysmall} relating ideals of small type and Betti splittings of powers of their mixed sums. This will be useful to studying invariants of powers of mixed sums in Sections \ref{sect_depthreg} and \ref{sect_lindofpowers}.
\subsection{Differential criteria}
In this subsection, let $R=k[x_1,\ldots,x_m]$ (where $m\ge 0$) be a polynomial ring, $\mm$ its graded maximal ideal, and $(0) \neq I \subseteq \mm$ a homogeneous ideal. We denote by $\partial(I)$ the ideal generated by elements of the form $\partial f/\partial x_i$, where $f\in I, i=1,\ldots,m$. In the proof of Proposition \ref{prop_stronglyGolodisofdoublysmalltype}, we will make use of the following criterion for detecting Tor-vanishing homomorphisms.
\begin{lem}[Ahangari Maleki {\cite[Proposition 3.5]{A}}]
\label{lem_mapsofTor}
Assume that $\chara k=0$.
\begin{enumerate}[\quad \rm (i)]
\item Let $I_1$ and $I_2$ be homogeneous ideals of $R$ such that $\partial(I_1)\subseteq I_2$. Then $I_1\subseteq \mm I_2$ and the map $I_1\to I_2$ is Tor-vanishing.
\item In particular, any homogeneous ideal $I$ of $R$ is of small type.
\end{enumerate}
\end{lem}
The only new statement in Lemma \ref{lem_mapsofTor} is the inclusion $I_1 \subseteq \mm I_2$, which results from Euler's identity for homogeneous polynomials. 

We wish to prove that every proper monomial ideal of a polynomial ring is of small type. For this, we record the following simple but very useful lemma. It is an application of the Taylor resolution \cite{Ta}; see \cite[Section 7.1]{HH2}. Below, for a set $G'$ of polynomials in $R$, denote $\lcm G'=\lcm (x: x\in G')$. If $I$ is a monomial ideal, let $\Gc(I)$ denote the set of its minimal monomial generators.
\begin{lem}[Eliahou and Kervaire {\cite[Proof of Proposition 3.1]{EK}}]
\label{lem_mapsofgens_LCM}
Let $(0) \neq I_1 \subseteq I_2$ be monomial ideals of $R$. Assume that there exists a function $\phi: \Gc(I_1) \to \Gc(I_2)$ with the following property:
\begin{enumerate}
 \item[\textup{(LCM)}] For any non-empty subset $G'$ of $\Gc(I_1)$, $\lcm G'$ belongs to $(\lcm \phi(G'))\mm$.
\end{enumerate}
Then the inclusion map $I_1\to I_2$ is Tor-vanishing.
\end{lem}
For $I$ being a monomial ideal of $R$, we denote by $\partial^*(I)$ the ideal generated by elements of the form $f/x_i$, where $f$ is a minimal monomial generator of $I$ and $x_i$ is a variable dividing $f$. The following lemma is straightforward so we leave the detailed proof to the interested reader. 
\begin{lem}
\label{lem_property_partial*}
Let $I_1, I_2, L$ be monomial ideals of $R$ where $L\subseteq I_1$. Then the following statements hold:
\begin{enumerate}[\quad \rm(i)]
\item $\partial^*(I_1)$ is a monomial ideal and $\partial(I_1) \subseteq \partial^*(I_1)$. 
\item $I_1 \subseteq \mm \partial^*(I_1)$.
\item Let $g\ge 1$ be maximal such that there exists a variable $x_i$ of $R$ with the property that $x_i^g$ divides an element of $\Gc(I_1)$. If $\chara k=0$ or $\chara k>g$ then $\partial^*(I_1)=\partial(I_1)$. In particular, if $I_1$ is squarefree then $\partial^*(I_1)=\partial(I_1)$.
\item $\partial^*(L) \subseteq \partial^*(I_1)$.
\item $\partial^*(I_1I_2)=\partial^*(I_1)I_2+I_1\partial^*(I_2)$.
\item $\partial^*(I^s)=\partial^*(I)I^{s-1}$ for all $s\ge 1$.
\end{enumerate}
\end{lem}
We have the following criterion for maps between sets of minimal generators of monomial ideals to have the (LCM) property.
\begin{prop}
\label{prop_criterion_partial*}
Let $I_1, I_2$ be non-zero proper monomial ideals of $R$ such that $\partial^*(I_1) \subseteq I_2$. Then $I_1\subseteq \mm I_2$ and there exists a map $\phi: \Gc(I_1) \to \Gc(I_2)$ having the property \textup{(LCM)} of Lemma \ref{lem_mapsofgens_LCM}. 
\end{prop}

\begin{proof}
From Lemma \ref{lem_property_partial*}(ii), we have $I_1 \subseteq \mm\partial^*(I_1) \subseteq \mm I_2$.

For the remaining statement, we use induction on the number of minimal generators of $I_1$. For a monomial $f\in R$, denote by $\supp(f)$ the set of variables dividing $f$. If $I_1$ is a principal ideal $(f)$, then $\partial^*(I_1)=(f/x_i: x_i \in \supp(f))$. By the hypothesis, $f/x_i \in I_2$ for some $x_i\in \supp(f)$, so there exists $g\in \Gc(I_2)$ dividing $f/x_i$. We define $\phi$ by $\phi(f)=g$, then clearly $f\in g\mm$. 

Assume that $|\Gc(I_1)|\ge 2$. Let $x$ be a variable which divides a minimal monomial generator of $I_1$. We can write in a unique way $I_1=xK+L$, where $L$ is generated by the elements of $\Gc(I_1)$ which are not divisible by $x$ and $xK$ is generated by the remaining elements.

Observe that $\Gc(xK)\cap \Gc(L)=\emptyset$, and $\Gc(I_1) =\Gc(xK) \cup \Gc(L)$. Furthermore $K\subseteq \partial^*(I_1) \subseteq I_2$.

Firstly, assume that $L \neq (0)$. Since $|\Gc(L)|<|\Gc(I_1)|$ and $\partial^*(L) \subseteq \partial^*(I_1) \subseteq I_2$ thanks to Lemma \ref{lem_property_partial*}(iv), by induction hypothesis, there exists a function $\psi: \Gc(L) \to \Gc(I_2)$ which has the property (LCM).

We define $\phi: \Gc(I_1) \to \Gc(I_2)$ as follows: if $y\in \Gc(xK)$, since $K\subseteq I_2$, for a choice of monomials $f\in R, g\in \Gc(I_2)$ such that $y=xfg$, we let $\phi(y)=g$. If $y\in \Gc(L)$ then we set $\phi(y)=\psi(y)$. We verify that $\phi$ has the property (LCM).

Consider a set $G'=\{y_1,\ldots,y_r,z_1,\ldots,z_s\} \subseteq \Gc(I_1)$ where $y_i \in \Gc(xK)$ and $z_j \in \Gc(L)$. We write $y_i=xf_ig_i$ where $g_i=\phi(y_i)\in \Gc(I_2)$. If $r=0$ then $\lcm G'=\lcm \{z_1,\ldots,z_q\}$ is strictly divisible by $\lcm \phi(G')=\lcm \psi(G')$ by the choice of $\psi$. If $r\ge 1$, denote $g=\lcm (\phi(y_1),\ldots,\phi(y_r))=\lcm(g_1,\ldots,g_r)$.
Denote $b=\lcm(z_1,\ldots,z_s), b'=\lcm (\psi(z_1),\ldots,\psi(z_s))$ then $b'$ divides $b$. Now $\lcm \phi(G')=\lcm(g,b')$ and $\lcm G'$ is divisible by $\lcm(xg_1,\ldots,xg_r, z_1,\ldots,z_s)=\lcm(xg,b)=x\lcm(g,b)$, where the second equality holds since $\gcd(x,b)=1$. Since $b'$ divides $b$, we conclude that $\lcm \phi(G')$ strictly divides $\lcm G'$, as desired.

It remains to consider the case $L=0$. In this case $I_1=xK$, so $K\subseteq \partial^*(I_1) \subseteq I_2$. Define $\phi$ in the same way as above, we get the desired conclusion.
\end{proof}
An immediate consequence of Proposition \ref{prop_criterion_partial*} is the following.
\begin{thm}
\label{thm_monomialidealsareofsmalltype}
Any proper monomial ideal of $R$ is of small type.
\end{thm}
\begin{proof}
Let $I \subseteq \mm$ be a monomial ideal. There is nothing to do if $I=(0)$, so we assume that $I\neq (0)$. By Lemma \ref{lem_mapsofgens_LCM}, it suffices to prove that for any $s\ge 1$, there exists a map $\Gc(I^s)\to \Gc(I^{s-1})$ with the (LCM) property. Thanks to Proposition \ref{prop_criterion_partial*}, we only need to check that $\partial^*(I^s) \subseteq I^{s-1}$, but this follows from Lemma \ref{lem_property_partial*}(vi).
\end{proof}

\subsection{A catalog}

We provide in this section more classes of ideals which are of small or doubly small type. For some classes of ideals satisfying the hypothesis of the next result, see, e.g., \cite{BC}, \cite{HHO}, \cite{HHZ}.
\newpage
\begin{prop}
\label{prop_idealswithKoszulpowers}
Assume that all the powers of $I$ are Koszul. Then:
\begin{enumerate}[\quad \rm(i)]
\item $I$ is of small type.
\item  If moreover $R$ is a Koszul algebra and $I \subseteq \mm^2$, then $I$ is of doubly small type.
\end{enumerate}
\end{prop}
\begin{proof}
Take $r\ge 1$. Since $I^r\subseteq \mm I^{r-1}$ and $I^{r-1}$ is Koszul, the map $I^r\to I^{r-1}$ is Tor-vanishing, thanks to Theorem \ref{thm_m-smallext}.

Now assume further that $R$ is a Koszul algebra and $I \subseteq \mm^2$. For $r\ge 1$, consider the chain $I^r \subseteq \mm I^{r-1} \subseteq I^{r-1}$. Since $I^{r-1}$ is Koszul and $R$ is a Koszul algebra, we obtain by \cite[Corollary 3.8]{Ng1} that $\mm I^{r-1}$ is Koszul. From the inclusion $I^r\subseteq \mm^2I^{r-1}$ and Theorem \ref{thm_m-smallext}, the map $I^r\to \mm I^{r-1}$ is Tor-vanishing. As seen above, the map $\mm I^{r-1} \to I^{r-1}$ is Tor-vanishing as well, so Lemma \ref{lem_composition_Tor-vanishing} implies that $I^r\to I^{r-1}$ is doubly Tor-vanishing for all $r\ge 1$. In other words, $I$ is of doubly small type.
\end{proof}

\begin{prop}
Let $(R,\mm)$ be a noetherian local ring and $I$ a proper ideal generated by a regular sequence. Then:
\begin{enumerate}[\quad \rm (i)]
\item The ideal $I$ is of small type. 
\item If moreover $I$ is contained in $\mm^2$ then it is of doubly small type.
\end{enumerate}
\end{prop}
\begin{proof}
(i) We can use the Eagon--Northcott resolution to directly construct a suitable lifting of the map $I^r\to I^{r-1}$. Here is an alternative argument. By assumption $I=(f_1,\ldots,f_p)$ where $p\ge 1$ and $f_1,\ldots,f_p$ is a regular sequence of elements in $\mm$. We use induction on $p\ge 1$ and $r\ge 1$ that there exists a lifting on the level of minimal free resolutions of the map $I^r\to I^{r-1}$ which induces the zero map after tensoring with $k$. If $p=1$ or $r=1$ then the conclusion is clear. Assume that $p\ge 2$ and $r\ge 2$. Denote $K=(f_2,\ldots,f_p)$ and $f=f_1$. We have $I^r=fI^{r-1}+K^r.$ Since $f$ is $R/K^r$-regular (see \cite[Page 6]{BH}), we have $fI^{r-1}\cap K^r=fK^r.$

Let $A, B$ be the minimal free resolution of $K^r, I^{r-1}$, resp. Using the induction hypothesis, the map $K^r\to K^{r-1}$ is Tor-vanishing, hence so is $K^r\to I^{r-1}$. Since $f$ is $R/K^r$-regular, a lifting of the map $fK^r\to K^r$ is given by $A \xrightarrow{\cdot f} A$. Therefore a minimal free resolution of $I^r$ is obtained from the mapping cone construction for the map $fK^r \to fI^{r-1}\oplus K^r$. 
\begin{equation*}
\begin{gathered}
\xymatrix{
0 \ar[r] & fK^r \ar[r]  & fI^{r-1} \oplus K^r \ar[r]  & I^r \ar[d] \ar[r] & 0\\
          &  &   & I^{r-1}   &                          }
\end{gathered}
\end{equation*}
Let $\veps$ be a lifting of the map $K^r\to I^{r-1}$ such that $\veps \otimes_R k=0$. We have the following lifting diagram.
\begin{equation*}
\begin{gathered}
\xymatrix{
 A \ar[rrr]^{(\veps, \cdot f)} \ar[d]  &&& B\oplus A \ar[rrr]^{(\cdot f) \mid_B - \veps \mid_A} \ar[d] &&& B \ar[d] \\                       
 fK^r \ar[rrr]^{a\mapsto (a,a)}  &&& fI^{r-1} \oplus K^r \ar[rrr]^{(b,c) \mapsto b-c}  &&& I^{r-1} }
\end{gathered}
\end{equation*}
Specifically, a lifting $\Lambda: A\to B\oplus A$ of the map $fK^r\to fI^{r-1}\oplus K^r$ is given by
\[
A\ni x \mapsto (\veps(x),fx) \in B\oplus A.
\]
A lifting $\Phi: B\oplus A \to B$ of the map $fI^{r-1} \oplus K^r \to I^{r-1}$ is given by
\[
B\oplus A \ni (y,z) \mapsto fy -\veps(z) \in B.
\]
The composition $\Phi\circ \Lambda$ is zero, hence we can extend $\Phi$ to a map from the mapping cone of $\Lambda$ to $B$ by setting it to be zero on $A$. The extended map is a lifting of the map $I^r\to I^{r-1}$ which is zero after tensoring with $k$. Hence $I^r \to I^{r-1}$ is Tor-vanishing. The induction and hence the proof is finished.

(ii) Argue similarly as for (i).
\end{proof}

Assume that $R=k[x_1,\ldots,x_m]$ is a polynomial ring over $k$ (where $m\ge 0$) and $(0)\neq I \subseteq R$ is a proper homogeneous ideal. Following Herzog and Huneke \cite{HeHu}, if $\chara k=0$ we say that a homogeneous (but possibly non-monomial) ideal $I$ is {\it strongly Golod} if $\partial(I)^2\subseteq I$. Independently of $\chara k$, we say that an ideal $I$ is $\up{*}$strongly Golod if $I$ is a monomial ideal and $\partial^*(I)^2\subseteq I$. The two notions are compatible when $\chara k=0$: $I$ is $\up{*}$strongly Golod if and only if $I$ is strongly Golod in the sense of Herzog and Huneke. Below, let $\widetilde{I}=\bigcup_{s\ge 1} (I:\mm^s)$ be the saturation of $I$. Let $\overline{I}$, the integral closure of $I$, be the set of $x\in R$ which satisfies a relation $x^n+a_1x^{n-1} +\cdots+a_{n-1}x+a_n=0$ where $n\ge 1, a_i\in I^i$ for $1\le i \le n$. The integral closure is an ideal containing $I$, and if $I$ is a monomial ideal then so is $\overline{I}$: In fact, by Proposition 1.4.2 and the discussion preceding it in \cite{HSw}, $\overline{I}$ is generated by monomials $f\in R$ such that $f^r\in I^r$ for some $r\ge 1$.

For each $s\ge 1$, let $I^{(s)}=\bigcap_{P\in \Min(I)}I^sR_P \cap R$ be the $s$-th symbolic power of $I$, where $\Min(I)$ denotes the set of minimal primes of $I$.
\begin{prop}
\label{prop_stronglyGolodisofdoublysmalltype}
Assume that either $\chara k=0$ or $I$ is a monomial ideal. 
\begin{enumerate}[\quad \rm (i)]
\item If $I$ is strongly Golod or $\up{*}$strongly Golod then it is of doubly small type.
\item  For all $s\ge 2$, the ideals $I^s, \widetilde{I^s}, I^{(s)}$ are of doubly small type. 
\item If $I$ is monomial then for all $s\ge 2$, $\overline{I^s}$ is of doubly small type.
\end{enumerate}
\end{prop}
First, we establish the following lemma, which is inspired by \cite[Theorem 2.3 and Proposition 3.1]{HeHu}. The proof of \cite[Proposition 3.1]{HeHu} in page 96--97 contains several problems/mistakes, e.g., in p.\,97, line 5, the best one can say is $(w/x_k)^r \in I^aI^{\lfloor (b-a)/2\rfloor}$. In p.\,97, line 7, the equality $r=d+b$ is generally false, one can only say $r\le d+b$. The crucial claim in the proof is nevertheless correct. We salvage the above problems and mistakes in the proof of part (i) of Lemma \ref{lem_*stronglyGolod}.
\begin{lem}
\label{lem_*stronglyGolod}
Let $I$ be a monomial ideal.
\begin{enumerate}[\quad \rm(i)]
\item If $I$ is $\up{*}$strongly Golod, then so is $\overline{I}$.
\item Let $I$ be $\up{*}$strongly Golod, and $L$ a monomial ideal such that $I:L=I:L^2$. Then $I:L$ is also $\up{*}$strongly Golod. 
\item For all $s\ge 2$, $I^s$ is $\up{*}$strongly Golod.
\item For all $s\ge 2$, $\overline{I^s}, I^{(s)}, \widetilde{I^s}$ are $\up{*}$strongly Golod.
\end{enumerate}
\end{lem}
\begin{proof}
(i) Let $f$ be a monomial such that $f^r \in I^r$ for some $r\ge 1$ and $x_i\in \supp(f)$. We claim that  $(f/x_i)^r\in I^{\lfloor r/2 \rfloor}$.

Write $f^r=m_1\cdots m_r$ where for each $1\le j\le r$, $m_j$ is a monomial in $I$. For $1\le j\le r$, let $d_j$ be maximal so that $x_i^{d_j}$ divides $m_j$. By permuting the indices, we can assume that $d_1=\cdots=d_a=0$ and $1\le d_{a+1}\le \cdots \le d_r$. Clearly $x_i^r|f^r$, hence $d_{a+1}+\cdots+d_r \ge r$. If $a\ge \lfloor r/2 \rfloor$, then
\[
(f/x_i)^r= (m_1\cdots m_a) \left((m_{a+1}\cdots m_r)/x_i^r\right) \in I^a \subseteq I^{\lfloor\frac{r}{2} \rfloor}.
\]
Consider the case $a< \lfloor r/2 \rfloor$. 

{\bf Observation}: It holds that $d_{r-a+1}+\cdots+d_r \ge 2a$. 

Indeed, assume that $d_{r-a+1}+\cdots+d_r < 2a$, then since $d_{r-a+1} \le \cdots \le d_r$, we get $d_{r-a+1} <2$. In particular, $d_{a+1}=\cdots=d_{r-a}=1$. This yields $d_{a+1} +\cdots+ d_r =r-2a+ (d_{r-a+1}+\cdots+d_r) <r$,
which is a contradiction. Hence the observation is true.

Since $I$ is $\up{*}$strongly Golod and $x_i$ divides $m_{a+1},\ldots,m_{r-a}$, we have an inclusion $(m_{a+1}/x_i)\cdots (m_{r-a}/x_i) \in I^{\lfloor (r-2a)/2\rfloor}$. Together with the observation, we get the inclusion in the following chain
\[
\left(\frac{f}{x_i}\right)^r=(m_1\cdots m_a)\left(\frac{m_{a+1}}{x_i} \cdots \frac{m_{r-a}}{x_i}\right) \frac{m_{r-a+1}\cdots m_r}{x_i^{2a}} \in I^aI^{\lfloor\frac{r-2a}{2}\rfloor}= I^{\lfloor\frac{r}{2} \rfloor}.
\]
This finishes the proof of the claim. 

Now take $f, g\in \overline{I}$ and $x_i\in \supp(f), x_j\in \supp(g)$. There exist $r,s\ge 1$ such that $f^r \in I^r,g^s\in I^s$. Note that $f^{2rs}, g^{2rs} \in I^{2rs}$, hence by the above claim we obtain $(fg/(x_ix_j))^{2rs}=(f/x_i)^{2rs}(g/x_j)^{2rs} \in I^{rs}I^{rs}=I^{2rs}$. Hence $(f/x_i)(g/x_j) \in \overline{I}$. This implies that $\overline{I}$ is $\up{*}$strongly Golod.

(ii) Take $f, g\in I:L$ and $x_i \in \supp(f), x_j\in \supp(g)$. Take any $h_1,h_2\in L$. Then $fh_1,gh_2\in I$, so $(f/x_i)h_1,(g/x_j)h_2 \in \partial^*(I)$. In particular $(f/x_i)(g/x_j)h_1h_2 \in \partial^*(I)^2 \subseteq I$. The last chain together with the hypothesis yields $(f/x_i)(g/x_j) \in I:L^2=I:L$. This implies that $\partial^*(I:L)^2 \subseteq I:L$, namely $I:L$ is $\up{*}$strongly Golod.

(iii) By Lemma \ref{lem_property_partial*}(vi), $\partial^*(I^s)=\partial^*(I)I^{s-1}$. Hence $\partial^*(I^s)^2 \subseteq (I^{s-1})^2 \subseteq I^s$, proving that $I^s$ is $\up{*}$strongly Golod.

(iv) By (i) and (iii), we have that $\overline{I^s}$ is $\up{*}$strongly Golod.

Let $J$ be the intersection of all the associated, non-minimal prime ideals of $I^s$. Then $I^{(s)}=\cup_{i\ge 1}(I^s:J^i)$. Since $R$ is noetherian, for large enough $i$, $I^{(s)}=I^s:J^i$. Set $L=J^i$, then $I^{(s)}=I^s:L=I^s:L^2$. Hence by (ii) and (iii), $I^{(s)}$ is $\up{*}$strongly Golod.

Similarly from (ii) and (iii), we deduce that $\widetilde{I^s}$ is $\up{*}$strongly Golod. The proof is concluded.
\end{proof}
\begin{proof}[Proof of Proposition \ref{prop_stronglyGolodisofdoublysmalltype}]
(i) First, consider the case $\chara k=0$. Let $P=\partial(I)I^r$ then $I^{r+1} \subseteq P \subseteq I^r$ by Euler's identity for homogeneous polynomials. By Lemma \ref{lem_composition_Tor-vanishing}, it suffices to show that the inclusion maps $I^{r+1} \subseteq P$ and $P\subseteq I^r$ are Tor-vanishing. For this, we wish to apply Lemma \ref{lem_mapsofTor}, so we claim that $\partial(I^{r+1}) \subseteq P$ and $\partial(P) \subseteq I^r$. The first inclusion is clear. Now $\partial(P) \subseteq \partial(\partial(I))I^r+\partial(I)^2I^{r-1} \subseteq I^r$,
where the second inclusion follows from the hypothesis.

The case $I$ is $\up{*}$strongly Golod is proved similarly, where $\partial^*(I)$ and Proposition \ref{prop_criterion_partial*} are used in places of $\partial(I)$ and Lemma \ref{lem_mapsofTor}. 

(ii) If $\chara k=0$, by \cite[Theorem 2.3]{HeHu}, for any $s\ge 2$, the ideals $I^s, \widetilde{I^s}, I^{(s)}$ are strongly Golod. If $I$ is monomial, by Lemma \ref{lem_*stronglyGolod}, the same ideals are $\up{*}$strongly Golod. Together with part (i) we get that such ideals are of doubly small type.

(iii) This follows by combining Lemma \ref{lem_*stronglyGolod}(iv) and part (i).
\end{proof}
\subsection{Powers of mixed sums}
The original purpose of introducing ideals of small type is contained in part (i) of the following result. The purpose of introducing ideals of doubly small type is contained in Theorem \ref{thm_boundld_mixedsum}(ii).
\newpage
\begin{thm} 
\label{thm_sumsofsmalldoublysmall} 
Let $(R,\mm)$ and $(S,\nn)$ be standard graded algebras over $k$. Let $I\subseteq R, J\subseteq S$ be non-zero proper homogeneous ideals. Denote $T=R\otimes_k S$ and $P=I+J \subseteq T$. 
\begin{enumerate}[\quad \rm (i)]
\item Assume that $I$ and $J$ are of small type. Then for all $r\ge 0$ and all $s\ge 1$, we have a Betti splitting $I^rP^s=I^{r+1}P^{s-1}+I^rJ^s$. Furthermore, $P$ is also of small type.
\item Assume that $I$ and $J$ are of doubly small type. Then for all $r\ge 0$ and all $s\ge 1$, the maps $I^{r+1}J^s\to I^rJ^s$ and $I^{r+1}J^s\to I^{r+1}P^{s-1}$ are doubly Tor-vanishing, and the ideal $P$ is of doubly small type.
\end{enumerate}
\end{thm}
\begin{proof}
(i) By Lemma \ref{lem_intersect_quotient} we have $I^{r+1}P^{s-1}\cap I^rJ^s=I^{r+1}J^s$. By Lemma \ref{lem_criterion_Bettisplit}, it suffices to check that the maps $I^{r+1}J^s \to I^rJ^s$ and $I^{r+1}J^s \to I^{r+1}P^{s-1}$ are Tor-vanishing. Since $I$ is of small type, the map $I^{r+1}\to I^r$ is Tor-vanishing. Tensoring over $k$ with $J^s$ is an exact functor, so the map $I^{r+1}J^s \to I^rJ^s$ is Tor-vanishing. The same holds for the map $I^{r+1}J^s \to I^{r+1}J^{s-1}$, and hence also for $I^{r+1}J^s \to I^{r+1}P^{s-1}$ which factors through the previous map.

We prove that for every $s\ge 1$, the map $I^rP^s\to I^rP^{s-1}$ is Tor-vanishing for all $r\ge 0$.  This clearly implies that $P$ is of small type. Induct on $s$. 

\textbf{Step 1}: If $s=1$, using $I^rP=I^{r+1}+I^rJ$ and Lemma \ref{lem_intersect_quotient}, we have the exactness of the first row in the following diagram
\begin{equation*}
\begin{gathered}
\xymatrix{
0 \ar[r] & I^{r+1}J \ar[r]  & I^{r+1} \oplus I^rJ \ar[r]  & I^rP \ar[d] \ar[r] & 0\\
         &                  &                             & I^r                & }
\end{gathered}
\end{equation*}
Let $\up{r}F, \up{r}G, \up{r}H$ be the minimal graded free resolutions of $I^r, J^r, P^r$ over $R,S, T$, resp. Then  minimal graded free resolutions over $T$ of $I^{r+1}J, I^{r+1} \oplus I^rJ$ and $I^r$ are given by $A=\up{r+1}F\otimes_k \up{1}G, B \oplus \overline{B}=(\up{r+1}F\otimes_k S) \oplus (\up{r}F\otimes_k \up{1}G)$ and $C=\up{r}F\otimes_k S$, resp. 

Let $\vphi: \up{r+1}F\to \up{r}F$ be a lifting of the map $I^{r+1} \to I^r$. Since $I$ is of small type, we can choose $\vphi$ such that $\vphi \otimes_R k=0$. Let $\veps: \up{1}G \to S$ be a lifting of the map $J \to S$ given by $\veps_i=0$ for $i\ge 1$. Choose bases for $\up{r+1}F, \up{r}F, \up{1}G$. We have a lifting diagram as follows.
\begin{equation*}
\begin{gathered}
\xymatrix{
 A \ar[rrr]^{(\id \otimes \veps, \vphi \otimes \id)} \ar[d]  &&& B\oplus \overline{B} \ar[rrr]^{(\vphi\otimes \id) \mid_B - (\id \otimes \veps)\mid_{\overline{B}}} \ar[d] &&& C \ar[d] \\                       
 I^{r+1}J \ar[rrr]^{a\mapsto (a,a)}  &&& I^{r+1} \oplus I^rJ \ar[rrr]^{(b,c) \mapsto b-c}  &&& I^r }
\end{gathered}
\end{equation*}
In other words, a lifting of $\Lambda: A\to B\oplus \overline{B}$ of the map $I^{r+1}J \to I^{r+1} \oplus I^rJ$ is given on the basis elements by 
$$
A\ni x\otimes y \mapsto (x\otimes \veps(y),\vphi(x)\otimes y) \in B\oplus \overline{B}.
$$ 
A lifting $\Pi: B\oplus \overline{B} \to C$ of the map $I^{r+1} \oplus I^rJ \to I^r$ is given on the basis elements by 
$$
B\oplus \overline{B} \ni (z\otimes 1, u\otimes v) \mapsto \vphi(z)\otimes 1 - u\otimes \veps(v) \in C.
$$
In particular, the composition map $A\to C$ is zero. Therefore we can extend $\Pi: B\oplus \overline{B} \to C$ to a map on the mapping cone of $\Lambda$ by setting $\Pi|_A =0$. By the choice of $\vphi$ and $\veps$, we have $\Pi \otimes_T k=0$. Hence $I^rP\to I^r$ is Tor-vanishing.

\textbf{Step 2}: Now assume that $s\ge 2$. Thanks to the argument for the Betti splitting $I^rP^s=I^{r+1}P^{s-1}+ I^rJ^s$ from above, we have a commutative diagram with exact rows. 
\begin{equation*}
\begin{gathered}
\xymatrixcolsep{5mm}
\xymatrixrowsep{5mm}
\xymatrix{
0 \ar[rr] && I^{r+1}J^s \ar[rr]^{\up{1}\omega} \ar[d]^{\gamma} && I^{r+1}P^{s-1} \oplus I^rJ^s \ar[rr] \ar[d]^{\delta} && I^rP^s \ar[d] \ar[rr] && 0\\
0 \ar[rr] && I^{r+1}J^{s-1} \ar[rr]^{\up{2}\omega} && I^{r+1}P^{s-2} \oplus I^rJ^{s-1} \ar[rr]  && I^rP^{s-1}  \ar[rr] && 0                             }
\end{gathered}
\end{equation*}
Let $\up{1}A,\up{1}B$, and $\up{1}C$ be the minimal graded free resolutions of $I^{r+1}J^s,I^{r+1}P^{s-1}$, and $I^rJ^s$, resp. Let $\up{2}A,\up{2}B$, and $\up{2}C$ be the minimal graded free resolution of $I^{r+1}J^{s-1}$, $I^{r+1}P^{s-2}$, and $I^rJ^{s-1}$, resp. With the notations as above, we can choose $\up{1}A=\up{r+1}F\otimes_k \up{s}G, \up{1}C=\up{r}F\otimes_k \up{s}G, \up{2}A=\up{r+1}F\otimes_k \up{s-1}G$ and finally $\up{2}C=\up{r}F\otimes_k \up{s-1}G$. 

% \Omega -- Delta
% \Phi -- \up{1}\Omega
% \Pi -- \up{2}\Omega
% \Psi -- Gamma

\textsf{Step 2a}: We show that there exist liftings $\up{1}\Omega, \up{2}\Omega, \Gamma, \Delta$ of $\up{1}\omega, \up{2}\omega, \gamma,\delta$ such that the following two maps $\up{1}A \to \up{2}B\oplus \up{2}C$ are equal: $\Delta\circ \up{1}\Omega=\up{2}\Omega \circ \Gamma$.
\begin{equation*}
\begin{gathered}
\xymatrixcolsep{5mm}
\xymatrixrowsep{5mm}
\xymatrix{
 \up{1}A \ar@{.>}[rrd] \ar[rrrr]^{\up{1}\Omega} \ar[ddd]^{\Gamma} &&  && \up{1}B\oplus \up{1}C \ar[rrrr] \ar@{.>}[d] \ar@/^4pc/[ddd]^{\Delta} && && \up{1}D \ar@{.>}[lld] \ar[ddd] \\
&& I^{r+1}J^s \ar[rr]^{\up{1}\omega} \ar[d]^{\gamma} && I^{r+1}P^{s-1} \oplus I^rJ^s \ar[rr] \ar[d]^{\delta} && I^rP^s \ar[d]  &&\\
&& I^{r+1}J^{s-1} \ar[rr]^{\up{2}\omega} && I^{r+1}P^{s-2} \oplus I^rJ^{s-1} \ar[rr]  && I^rP^{s-1}  &&    \\
 \up{2}A \ar@{.>}[rru] \ar[rrrr]^{\up{2}\Omega}  && &&    \up{2}B \oplus \up{2}C \ar[rrrr] \ar@{.>}[u]    &&    &&    \up{2}D \ar@{.>}[llu]               }
\end{gathered}
\end{equation*}
Let $\pp$ be the graded maximal ideal of $T$. Since $J$ is of small type, there is a lifting $\Gamma: \up{1}A\to \up{2}A$ of $\gamma$ such that $\Gamma(\up{1}A) \subseteq \pp (\up{2}A)$. In more details, choose a lifting $\eta_s:\up{s}G \to \up{s-1}G$ of $J^s\to J^{s-1}$ so that $\eta_s\otimes_S k=0$ and choose $\Gamma=\id\otimes_k \eta_s$. 

We will define $\up{1}\Omega, \Delta, \up{2}\Omega$ componentwise; see the next two diagrams. Let $\up{1}\omega_1: I^{r+1}J^s\to I^{r+1}P^{s-1}, \up{1}\omega_2: I^{r+1}J^s \to I^rJ^s$ be the restrictions of $\up{1}\omega$ to the first and second component of the image, resp. Similarly define $\up{2}\omega_1, \up{2}\omega_2$. Let $\delta_1:I^{r+1}P^{s-1}\to I^{r+1}P^{s-2}, \delta_2: I^rJ^s \to I^rJ^{s-1}$ be the components of $\delta$. By the induction hypothesis, $\delta_1$ is also Tor-vanishing, hence there is a lifting $\Delta_1: \up{1}B \to \up{2}B$ of it such that $\Delta_1(\up{1}B) \subseteq \pp (\up{2}B)$. Let $\Theta$ be a lifting of the map $\theta: I^{r+1}J^{s-1} \to I^{r+1}P^{s-1}$, then we have liftings $\up{1}\Omega_1=\Theta \circ \Gamma$ of $\up{1}\omega_1$ and $\up{2}\Omega_1=\Delta_1 \circ \Theta$ of $\up{2}\omega_1$. Moreover, $\Delta_1 \circ \up{1}\Omega_1=\up{2}\Omega_1 \circ \Gamma$.
\[
\xymatrix{
 I^{r+1}J^s \ar[rr]^{\up{1}\omega_1} \ar[d]^{\gamma} && I^{r+1}P^{s-1}  \ar[d]^{\delta_1} \\
I^{r+1}J^{s-1} \ar[rr]^{\up{2}\omega_1} \ar[rru]^{\theta} && I^{r+1}P^{s-2} 
}
\]
Since $I$ is of small type, for each $r$, we can choose a lifting $\rho_r: \up{r}F \to \up{r-1}F$ of $I^r\to I^{r-1}$ such that $\rho_r \otimes_R k=0$. Let $\up{1}\Omega_2=\rho_{r+1}\otimes_k \id: \up{1}A \to \up{1}C, \Delta_2=\id\otimes_k \eta_s: \up{1}C \to \up{2}C$ and $\up{2}\Omega_2= \rho_{r+1}\otimes_k \id: \up{2}A \to \up{2}C$, which are liftings of $\up{1}\omega_2, \delta_2$ and $\up{2}\omega_2$, resp. It is immediate that $\Delta_2 \circ \up{1}\Omega_2=\up{2}\Omega_2 \circ \Gamma$.
\[
\xymatrix{
 I^{r+1}J^s \ar[rr]^{\up{1}\omega_2} \ar[d]^{\gamma} && I^rJ^s  \ar[d]^{\delta_2} \\
I^{r+1}J^{s-1} \ar[rr]^{\up{2}\omega_2}  && I^rJ^{s-1} 
}
\]
Define $\up{i}\Omega$ by $\up{i}\Omega_1$ and $\up{i}\Omega_2$ for $i=1,2$, $\Delta=(\Delta_1, \Delta_2)$. Then one has $\Delta\circ \up{1}\Omega=\up{2}\Omega \circ \Gamma$, as desired.

\textsf{Step 2b}: Let $\up{1}D, \up{2}D$ be the mapping cones of $\up{1}\omega, \up{2}\omega$, resp. By the construction of $\up{1}\omega$ and $\up{2}\omega$, $\up{1}D, \up{2}D$ are minimal graded free resolutions of $I^rP^s$, $I^rP^{s-1}$, resp. The liftings $\Gamma$ and $\Delta$ together with the commutativity relation $\Delta\circ \up{1}\Omega=\up{2}\Omega \circ \Gamma$ yields a lifting $\chi: \up{1}D\to \up{2}D$ for $I^rP^s\to I^rP^{s-1}$. Since $\Gamma \otimes_T k=0$ and $\Delta \otimes_T k=0$, we deduce that $\chi\otimes_T k=0$. Hence $I^rP^s\to I^rP^{s-1}$ is Tor-vanishing, as desired. This finishes the induction and the proof of part (i).

(ii) can be proved similarly as (i).
\end{proof}

%-------------------------------------------------------
%-------------------------------------------------------
%-------------------------------------------------------
%-------------------------------------------------------
%-------------------------------------------------------
%-------------------------------------------------------

\section{Projective dimension and regularity}
\label{sect_depthreg}
From now on, unless stated otherwise, we fix the following notations: Let $(R,\mm)$ and $(S,\nn)$ be standard graded algebras over $k$. Let $I\subseteq R, J\subseteq S$ be non-zero proper homogeneous ideals. Denote $T=R\otimes_k S$ and let $P=I+J \subseteq T$ be the mixed sum of $I$ and $J$.

The main technical result of this section is Proposition \ref{prop_formula_depthreg}. We will deduce easily from Proposition \ref{prop_formula_depthreg} the asymptotic formulas for the projective dimension and the (Castelnuovo--Mumford) regularity of $P^s$ for $s\gg 0$, recovering in part results of \cite{HTT} (see Theorems \ref{thm_asymptoticproj_polynomial} and \ref{thm_asymptoticreg}). 
\subsection{Formulas}
\begin{prop}
\label{prop_formula_depthreg}
Assume that $I$ and $J$ are of small type. Assume further that $I^i, J^i\neq 0$ for all $i\ge 1$, e.g., $R$ and $S$ are reduced. Then for all $s\ge 1$, we have
\begin{align*}
\projdim_T P^s&=\max_{i \in [1,s-1], j\in [1,s]}\left\{\projdim_R I^{s-i}+\projdim_S J^i, \projdim_R I^{s-j+1}+\projdim_S J^j+1\right\},\\
\reg_T P^s&=\max_{i\in [1,s-1], j\in [1,s]}\left\{\reg_R I^{s-i}+\reg_S J^i, \reg_R I^{s-j+1}+\reg_S J^j-1\right\}.
\end{align*}
\end{prop}
\begin{proof}
By the proof of Theorem \ref{thm_sumsofsmalldoublysmall}, the decomposition $I^rP^s=I^{r+1}P^{s-1}+I^rJ^s$ is a Betti splitting and $I^{r+1}P^{s-1}\cap I^rJ^s=I^{r+1}J^s$. Hence by Lemma \ref{lem_depthreg_Bettisplit}
\[
\projdim_T (I^rP^s)=\max \{\projdim_T (I^{r+1}P^{s-1}),\projdim_T(I^rJ^s),\projdim_T(I^{r+1}J^s)+1\}.
\]
Together with the isomorphism $IJ\cong I\otimes_k J$ and Lemma \ref{lem_tensor}, we conclude that
\[
\projdim_T (I^rP^s)=\max \{\projdim_T (I^{r+1}P^{s-1}),\projdim_R(I^r)+\projdim_S(J^s),\projdim_R(I^{r+1})+\projdim_S(J^s)+1\}.
\]
By induction on $s\ge 1$, we get
\[
\projdim_T (I^rP^s)=\max_{i\in [0,s], j\in [1,s]}\left\{\projdim_R I^{r+s-i}+\projdim_S J^i, \projdim_R I^{r+s-j+1}+\projdim_S J^j+1\right\}.
\]
Setting $r=0$, we reduce to
\[
\projdim_T P^s=\max_{i\in [0,s], j\in [1,s]}\left\{\projdim_R I^{s-i}+\projdim_S J^i, \projdim_R I^{s-j+1}+\projdim_S J^j+1\right\}.
\]
The term corresponding to $i=s$ can be omitted since it is smaller than the term corresponding to $j=s$. Similarly we can omit the term corresponding to $i=0$, since it is smaller than the term corresponding to $j=1$. Hence we receive the desired formula for $\projdim_T P^s$. The formula for $\reg_T P^s$ is proved similarly.
\end{proof}
\begin{ex}
The following example shows that the equality for regularity in Proposition \ref{prop_formula_depthreg} may fail if one of the ideals $I$ and $J$ is not of small type, even if both of them are not nilpotent. We take a cue from the example in \cite[Remark 2.7]{Co}. Take $R=\mathbb{Q}[a,b,c]$, $I=(a^4,a^3b,ab^3,b^4,a^2b^2c^5)$ and $S=\mathbb{Q}[x,y,z,t]/(x^2-yz,xy,y^2,z^2,xt,yt,zt)$, $J=(x,t)$. (By abuse of notations, we denote the residue class of $x$ in $S$ by $x$ itself, and so on.) We will show that the following hold:
\begin{enumerate}
\item $I$ is of small type, $J$ is neither of small type nor nilpotent.
\item $\reg_R I=9,\reg_S J = 2$.
\item $\reg_T P^2=10 < 11 \le \max\{\reg_R I+\reg_S J, \reg_R I^2+\reg_S J-1, \reg_R I+\reg_S J^2-1\}$.
\end{enumerate}

(i) By Theorem \ref{thm_monomialidealsareofsmalltype}, $I$ is of small type. Clearly $J^n=(t^n)$ for $n\ge 3$, hence $J$ is not nilpotent. By computation with Macaulay2 \cite{GS}, $\beta_{1,3}(J)\neq 0$ while $\beta_{1,3}(J/J^2)=0$, so the map $\Tor^S_1(k,J)\to \Tor^S_1(k,J/J^2)$ is not injective. This shows that $J$ is not of small type. 

(ii) As mentioned above, $\beta_{1,3}(J)\neq 0$, hence $\reg_S J\ge 2$. The following is a Koszul filtration for $S$ (see \cite{CTV} for this notion): 
$$
\left\{(0), (y), (y,x), (y,x,t), (y,x,z), (y,x,z,t)\right\}.
$$ 
Hence $S$ is a Koszul algebra by \cite[Proposition 1.2]{CTV}. Since $S$ is a Koszul algebra, by \cite[Theorem 1.1]{AE},
\[
\reg_S S/J \le \reg S/J=1.
\] 
This yields $\reg_S J=2$. By computations with Macaulay2, $\reg_R I=9$.

(iii) Denote $U=P^2+(t)$, $V=U+(x)=I^2+J$. By direct computations, $P^2:t=I+(x,y,z,t)$ and $U:x=I+(x,y,t)$. Hence there are exact sequences
\[
0\to P^2 \to P^2+(t)=U \to \frac{(t)}{(t)\cap P^2} \cong \left(\frac{T}{I+(x,y,z,t)}\right)(-1) \to 0.
\]
and
\[
0\to U \to U+(x)=V \to \frac{(x)}{(x)\cap U} \cong \left(\frac{T}{I+(x,y,t)}\right)(-1) \to 0.
\]
By elementary arguments, for our purpose, it suffices to show that 
\begin{enumerate}[\quad \rm (a)]
\item $\reg_T (I+(x,y,z,t))=9, \reg_T (I+(x,y,t))=9$,
\item $\reg_T V=\reg_T (I^2+J)=9$.
\end{enumerate}  
From the above Koszul filtration for $S$, $\reg_S (x,y,t)=\reg_S (x,y,z,t)=1$, so (a) follows from Proposition \ref{prop_invariants:mixedsum} and the fact that $\reg_R I=9$. For the same reason we get (b) from the fact that $\reg_R I^2=8$ and $\reg_S J=2$.

Unfortunately, we could not find an example showing that the condition $I$ and $J$ are of small type is necessary for the equality of projective dimension in Proposition \ref{prop_formula_depthreg}.
\end{ex}
We are now able to deliver the
\begin{proof}[Proof of Theorem \ref{thm_depthreg}]
Since $\chara k=0$ or $I$ and $J$ are monomial, by Lemma \ref{lem_mapsofTor}(ii) and Theorem \ref{thm_monomialidealsareofsmalltype}, $I$ and $J$ are of small type. Using the Auslander--Buchsbaum formula, Proposition \ref{prop_formula_depthreg}, and the equality $\dim T=\dim R+\dim S$, we see that
\[
\depth P^s=\min_{i \in [1,s-1], j\in [1,s]}\left\{\depth I^{s-i}+\depth J^i, \depth I^{s-j+1}+\depth J^j-1\right\}.
\]
Since $I, J, P\neq (0)$, we get from $\depth R/I=\depth I-1$ the equality (i). 

As $R$ is a polynomial ring, for any finitely generated graded $R$-module $M$, $\reg M=\reg_R M$ \cite[Proposition, Page 89]{EG}. Arguing similarly as for the equality of depth, using the equality $\reg R/I=\reg I-1$ and Proposition \ref{prop_formula_depthreg}, we get (ii).
\end{proof}
\begin{ex}
\label{ex_nonpolynomialbase}
It is natural to ask whether Theorem \ref{thm_depthreg} is still true if $R$ and $S$ are not both polynomial rings. The answer is ``no''. Consider $R=\mathbb{Q}[a,b,c], I=(a^4,a^3b,ab^3,b^4,a^2b^2c^5)$, $S=\mathbb{Q}[x,y,z]/(x^2,xz,yz)$ and $J=(x+z)$. Computations with Macaulay2 \cite{GS} show that:
\begin{enumerate}[\quad \rm(1)]
\item $\depth (T/P^2)=1>0 = \depth(R/I)+\depth(S/J^2)$. Hence the equality (i) for depth of Theorem \ref{thm_depthreg} does not hold.
\item $\reg(T/P^2)=9 < 10= \reg(R/I)+\reg(S/J)+1$. Hence the equality (ii) for Castelnuovo--Mumford regularity of Theorem \ref{thm_depthreg} does not hold.
\end{enumerate}
On the other hands, both equalities of Proposition \ref{prop_formula_depthreg} are true, thanks to Proposition \ref{prop_formula_projreg_specialcases} below. Indeed, since $J^n=((x+z)^n)$ and $x+z$ is a zero-divisor, $\projdim_S J^n=\projdim_S J=\infty$ for all $n\ge 1$. Furthermore, observe that $0:(x+z)=(x)$ and $0:(x+z)^n=0:z^n=(x,y)$ for $n\ge 2$. Since the collection of ideals generated by residue classes of variables forms a Koszul filtration for $S$, $\reg_S J^n=n$ for all $n\ge 1$.

We may see from this example why Proposition \ref{prop_formula_depthreg}, which deals with projective dimension and regularity over $T$, is a natural generalization of Theorem \ref{thm_depthreg}.
\end{ex}

In some special cases, the formulas of Proposition \ref{prop_formula_depthreg} hold with the assumption that only one of $I$ and $J$ is of small type. The following result is a generalization of \cite[Proposition 2.9]{HTT}; the proof of the latter cannot be adapted to our situation.
\begin{prop}
\label{prop_formula_projreg_specialcases}
Suppose that $I^i, J^i\neq 0$ for all $i\ge 1$.
\begin{enumerate}[\quad \rm(i)]
\item Assume that $J$ is of small type and $\projdim_S J^{i-1} \le \projdim_S J^i$ for all $i\ge 1$. Then for all $s\ge 1$, we have $\projdim_T P^s=\max_{i\in [1,s]}\left\{\projdim_R I^{i}+\projdim_S J^{s-i+1}+1\right\}$.
\item Assume that $J$ is of small type and $\reg_S J^{i-1} +1 \le \reg_S J^i$ for all $i\ge 1$. Then for all $s\ge 1$, there is an equality 
$$
\reg_T P^s=\max_{i\in [1,s]}\left\{\reg_R I^i+\reg_R J^{s-i+1}-1\right\}.
$$
\item In particular, if $\reg_S J^i=i$ for all $i\ge 1$ then for all $s\ge 1$, we have an equality $\reg_T P^s=\max_{i\in [1,s]}\left\{\reg_R I^i-i\right\}+s$.
\end{enumerate}
\end{prop}
\begin{proof}
Following \cite[Page 450]{Avr1}, we say that a map of graded $R$-modules $M\xrightarrow{\phi} P$ is {\it small} if $\Tor^R_i(k,\phi)$ is injective for all $i\ge 0$. First we claim that if $J$ is of small type then for all $r\ge 0, s\ge 1$, the inclusion map $I^{r+1}P^{s-1} \to I^rP^s$ is small. Consider the following diagram with exact rows and denote $U=(I^rJ^s)/(I^{r+1}J^s) \cong (I^r/I^{r+1})\otimes_k J^s$.
\begin{equation*}
\begin{gathered}
\xymatrixcolsep{5mm}
\xymatrixrowsep{5mm}
\xymatrix{
0 \ar[rr] && I^{r+1}J^s \ar[rr] \ar[d]^{\gamma} && I^rJ^s  \ar[rr] \ar[d] && \frac{I^rJ^s}{I^{r+1}J^s} \ar[d] \ar[rr] && 0\\
0 \ar[rr] && I^{r+1}P^{s-1} \ar[rr] && I^rP^s \ar[rr]  && \frac{I^rJ^s}{I^{r+1}J^s}  \ar[rr] && 0                             }
\end{gathered}
\end{equation*}
The long exact sequence of homology induces the following commutative rectangle.
\begin{equation*}
\begin{gathered}
\xymatrixcolsep{5mm}
\xymatrixrowsep{5mm}
\xymatrix{
\Tor^T_{i+1}(k,U) \ar[d]^{=} \ar[rr] && \Tor^T_i(k,I^{r+1}J^s) \ar[d] \\
\Tor^T_{i+1}(k,U) \ar[rr] && \Tor^T_i(k,I^{r+1}P^{s-1})   }
\end{gathered}
\end{equation*}
Since $J$ is of small type, similarly to the proof of Theorem \ref{thm_sumsofsmalldoublysmall}, the map $I^{r+1}J^s\to I^{r+1}P^{s-1}$ is Tor-vanishing, hence the diagram yields that the connecting map on the second row is zero. In other words, $I^{r+1}P^{s-1} \to I^rP^s$ is small, as claimed.

(i) Since $J$ is of small type, by the claim, we deduce
\begin{align*}
\projdim_T (I^rP^s)&=\max\{\projdim_T (I^{r+1}P^{s-1}),\projdim_T U\}\\
&=\max\left\{\projdim_T (I^{r+1}P^{s-1}),\projdim_R (I^r/I^{r+1})+\projdim_S J^s\right\}.
\end{align*}
The second equality follows from Lemma \ref{lem_tensor}. Using induction on $s\ge 1$, we then obtain
\[
\projdim_T (I^rP^s)=\max_{i\in [0,s-1]}\left\{\projdim_R (I^{r+i}/I^{r+i+1})+\projdim_S J^{s-i}, \projdim_R I^{r+s}\right\}.
\]
Setting $r=0$, we get $\projdim_T P^s=\max_{i\in [0,s-1]}\left\{\projdim_R (I^i/I^{i+1})+\projdim_S J^{s-i}, \projdim_R I^s\right\}.$ For the desired equality, it remains to show that
\begin{equation*}
\max_{i\in [0,s-1]}\left\{\projdim_R (I^i/I^{i+1})+\projdim_S J^{s-i}, \projdim_R I^s\right\}=\max_{i\in [1,s]}\left\{\projdim_R I^i+\projdim_S J^{s-i+1}+1\right\}.
\end{equation*}
Since $\projdim_R (R/I)=\projdim_R I+1$, the term corresponding to $i=0$ on the left-hand side equals the term corresponding to $i=1$ on the right-hand side. Noting that 
\begin{align*}
&\projdim_R (I^i/I^{i+1})+\projdim_S J^{s-i}\\
 &\le \max\{\projdim_T I^i+\projdim_S J^{s-i},\projdim_R I^{i+1}+\projdim_S J^{s-i}+1\}\\
&\le \max\{\projdim_T I^i+\projdim_S J^{s-i+1}+1,\projdim_R I^{i+1}+\projdim_S J^{s-i}+1\}.
\end{align*}
for $1\le i\le s-1$, we see that the left-hand side is $\le $ the right-hand side.

Conversely, using 
\begin{align*}
&\projdim_R I^i+\projdim_S J^{s-i+1}+1 \\
 &\le \max\{\projdim_R I^{i-1}+\projdim_S J^{s-i+1}+1, \projdim_R (I^{i-1}/I^i)+\projdim_S J^{s-i+1}\} \\
&\le \max\{\projdim_R I^{i-1}+\projdim_S J^{s-i+2}+1, \projdim_R (I^{i-1}/I^i)+\projdim_S J^{s-i+1}\}.
\end{align*}
for $2\le i\le s$, we get the reverse inequality. This finishes the proof of (i).

(ii) Argue similarly as for part (i).

(iii) is a consequence of Proposition \ref{prop_idealswithKoszulpowers} and (ii). The proof is concluded.
\end{proof}

\subsection{Asymptotes}
Define the index of projective dimension stability of $I$ to be $\pstab(I)=\min\{r\ge 1: \projdim_R I^i=\lim_{s\to \infty} \projdim_R I^s ~\text{for all $i\ge r$}\}$. This is a well-defined finite number, e.g. by a result of Kodiyalam \cite[Corollary 8]{K1}.  
\begin{thm}
\label{thm_asymptoticproj_general}
Assume that $I$ and $J$ are of small type and $I^i, J^i\neq 0$ for all $i\ge 1$. Then for all $s\ge \pstab(I)+\pstab(J)$, we have
\[
\projdim_T P^s=\max\left\{\lim_{i\to \infty}\projdim_R I^i+\max_{j \ge 1}\projdim_S J^j+1 , \max_{i \ge 1}\projdim_R I^i+\lim_{j\to \infty}\projdim_S J^j+1\right\}.
\]
\end{thm}
\begin{proof}
By Proposition \ref{prop_formula_depthreg} and its proof, we have
\[
\projdim_T P^s=\max_{i\in [1,s], j\in [1,s]}\left\{\projdim_R I^{s-i}+\projdim_S J^i,\projdim_R I^{s-j+1}+\projdim_S J^j+1\right\}.
\]
First we determine $\max_{i=1,\ldots,s}\left\{\projdim_R I^{s-i}+\projdim_S J^i\right\}$. If $\pstab(J)\le i\le s$ then $\projdim_S J^i=\lim_{j\to \infty} \projdim_S J^j$. For every such $i$, $0\le s-i\le s-\pstab(J)$ and $s-\pstab(J)\ge \pstab(I)$, hence
\[
\max_{\pstab(J)\le i\le s}\projdim_R I^{s-i}=\max_{i=0,\ldots,\pstab(I)}\projdim_R I^i=\max_{i\ge 1} \projdim_R I^i.
\]
So we obtain 
\[
\max_{i=\pstab(J),\ldots,s}\left\{\projdim_R I^{s-i}+\projdim_S J^i\right\}=\max_{i\ge 1} \projdim_R I^i+\lim_{j\to \infty} \projdim_S J^j.
\]
If $1\le i\le \pstab(J)$ then $s-i\ge \pstab(I)$, hence $\projdim_R I^{s-i}=\lim_{i\to \infty} \projdim_R I^i$. Thus
\begin{align*}
\max_{i=1,\ldots,\pstab(J)}\left\{\projdim_R I^{s-i}+\projdim_S J^i\right\}&=\lim_{i\to \infty} \projdim_R I^i+ \max_{j=1,\ldots,\pstab(J)}\projdim_S J^j \\
&=\lim_{i\to \infty} \projdim_R I^i+ \max_{j \ge 1}\projdim_S J^j.
\end{align*}
Summing up 
\begin{align*}
&\max_{i=1,\ldots,s}\left\{\projdim_R I^{s-i}+\projdim_S J^i\right\}=\\
&\max\left\{\max_{i\ge 1} \projdim_R I^i+\lim_{j\to \infty} \projdim_S J^j, \lim_{i\to \infty} \projdim_R I^i+ \max_{j \ge 1}\projdim_S J^j\right\}.
\end{align*}
Similarly, we have
\begin{align*}
&\max_{j=1,\ldots,s}\left\{\projdim_R I^{s-j+1}+\projdim_S J^j+1\right\}=\\
&\max\left\{\max_{i\ge 1} \projdim_R I^i+\lim_{j\to \infty}\projdim_S J^j+1, \lim_{i\to \infty} \projdim_R I^i+ \max_{j \ge 1}\projdim_S J^j+1\right\}.
\end{align*}
This yields the desired equality for $\projdim_T P^s$.
\end{proof}
The index of depth stability of $I$, denoted by $\dstab(I)$, is the least number $r$ such that $\depth I^i=\lim\limits_{s\to \infty} \depth I^s$ for all $i\ge r$. The next theorem recovers partly \cite[Theorem 4.6]{HTT} with an easier argument.
\begin{thm}
\label{thm_asymptoticproj_polynomial}
Assume that $R$ and $S$ are polynomial rings over $k$, and either $\chara k=0$ or $I$ and $J$ are monomial ideals. Then for all $s\ge \dstab(I)+\dstab(J)$, we have
\[
\projdim_T P^s=\max\left\{\lim_{i\to \infty}\projdim_R I^i+\max_{j \ge 1}\projdim_S J^j+1 , \max_{i \ge 1}\projdim_R I^i+\lim_{j\to \infty}\projdim_S J^j+1\right\}.
\]
\end{thm}
\begin{proof}
Since $R, S, T$ are regular rings, by the Auslander--Buchsbaum formula, we have equalities $\pstab(I)=\dstab(I), \pstab(J)=\dstab(J)$. Moreover, by Lemma \ref{lem_mapsofTor}(ii) and Theorem \ref{thm_monomialidealsareofsmalltype}, $I$ and $J$ are of small type. Hence the result follows from Theorem \ref{thm_asymptoticproj_general}.
\end{proof}

Assume that $R$ and $S$ are polynomial rings over $k$. By the result due of Kodiyalam \cite{K} and Cutkosky--Herzog--N.V. Trung \cite{CHT}, we can define the regularity stabilization index of $I$, denoted $\rstab(I)$, as follows
\[
\min\{r\ge 1: \text{there exist constants $e$ and $g$ such that}~ \reg I^s=es+g ~\text{for all $s\ge r$}\}.
\]
Note that in \cite{HTT}, the notation $\text{lin}(I)$ was used in place of $\rstab(I)$. The next result is a special case of \cite[Proposition 5.7]{HTT} but its proof is easier.
\begin{thm}
\label{thm_asymptoticreg}
Assume that $R$ and $S$ are polynomial rings over $k$. Assume in addition that either $\chara k=0$ or $I$ and $J$ are monomial ideals. Suppose that $\reg I^r=er+g$ for all $r\ge \rstab(I)$ and $\reg J^s=fs+h$ for all $s\ge \rstab(J)$, and moreover $e\ge f$. Denote $g^*=\max_{i=1,\ldots,\rstab(I)}(\reg I^i-fi)$ and $h^*=\max_{j=1,\ldots,\rstab(J)}(\reg J^j-ej)$. Then for all $s\ge \rstab(I)+\rstab(J)$, we have
\[
\reg P^s=\max\left\{e(s+1)+g+h^*,f(s+1)+g^*+h\right\}-1.
\]
\end{thm}
\begin{proof}[Sketch of proof]
By Lemma \ref{lem_mapsofTor}(ii) and Theorem \ref{thm_monomialidealsareofsmalltype}, $I$ and $J$ are of small type. Hence applying Proposition \ref{prop_formula_depthreg} and its proof, we have
\begin{equation}
\label{eq_regIrPsmax}
\reg P^s= \max_{i\in [1,s], j\in [1,s]}\left\{\reg I^{s-i}+\reg J^i,\reg I^{s-j+1}+\reg J^j-1\right\}.
\end{equation}
First we show that
\begin{equation}
\label{eq_max}
\max_{i=1,\ldots,s}\left\{\reg I^{s-i}+\reg J^i\right\}=\max\left\{es+g+h^*, fs+g^*+h\right\}.
\end{equation}
The argument runs along the line of the proof of Theorem \ref{thm_asymptoticproj_general}. Similarly to \eqref{eq_max}, we have
\begin{align}
\label{eq_max2}
&\max_{i=1, \ldots, s}\left\{\reg I^{s-i+1}+\reg J^i-1\right\}= \nonumber\\ 
&\max\left\{e(s+1)+g+h^*-1, f(s+1)+g^*+h-1\right\}.
\end{align}
Combining Equalities \eqref{eq_regIrPsmax}, \eqref{eq_max} and \eqref{eq_max2}, we get the desired equality. 
\end{proof}
\subsection{Applications}
\label{subsect_applications}
 In this subsection, assume that $R$ and $S$ are polynomial rings over $k$. Recall that the depth function of $I$ is given by $r\mapsto \depth(R/I^r)$ for $r\ge 1$. By Brodmann's theorem \cite{Br}, the sequence of values of any depth function is eventually constant. We say that $I$ {\it has a constant depth function} if $(\depth(R/I^r))_{r\ge 1}$ is a constant sequence. One of the main applications of \cite{HTT} is the following
\begin{prop}[H\`a, Trung and Trung, {\cite[Proposition 4.7]{HTT}}]
\label{prop_constantdepth}
Let $I$ and $J$ be squarefree monomial ideals of $R$ and $S$, resp. Then $I+J$ has a constant depth function if and only if both $I$ and $J$ do.
\end{prop}
For squarefree monomial ideals, the last result simplifies significantly the equivalence (i) $\Leftrightarrow$ (ii) of \cite[Theorem 1.1]{HV}, which has to require that the Rees algebras of $I$ and $J$ are Cohen--Macaulay. Thanks to Theorem \ref{thm_depthreg}, we obtain the following improvement of (the ``if'' part of) Proposition \ref{prop_constantdepth}. 
\begin{cor}
\label{cor_constantdepth} 
If either $\chara k=0$, or $I$ and $J$ are \textup{(}possibly non-squarefree\textup{)} monomial ideals, then $I+J$ has a constant depth function if $I$ and $J$ do. If $I$ and $J$ are squarefree monomial ideals, then conversely, if $I+J$ has a constant depth function then both $I$ and $J$ do.
\end{cor}
Our proof of the second part of this result is more straightforward than the original proof of \cite[Proposition 4.7]{HTT}. As noted in \cite[Example 4.8]{HTT}, the second statement of Corollary \ref{cor_constantdepth} does not hold if $I$ and $J$ are non-squarefree monomial ideals, even in characteristic zero.
\begin{proof}[Proof of Corollary \ref{cor_constantdepth}]
The first statement is immediate from Theorem \ref{thm_depthreg}(i).

For the second statement, we use the inequality $\depth R/\sqrt{L} \ge \depth R/L$ for $L$ being a monomial ideal of $R$ \cite[Proof of Theorem 2.6]{HeTT}. In particular, since $I$ and $J$ are squarefree monomial ideals, $\depth R/I \ge \depth R/I^s$ and $\depth S/J \ge \depth S/J^s$ for all $s\ge 1$. By Theorem \ref{thm_depthreg}(i), $\depth T/P=\depth R/I+\depth S/J$ and for all $n\ge 2$,

$\depth T/P^s=$
\begin{align*}
&\min_{i\in [1,s-1],j\in [1,s]} \left\{\depth R/I^i+\depth S/J^{s-i+1}+1,\depth R/I^j+\depth S/J^{s-j}\right\}\\
             & \le \min_{j\in [1,s]}\{\depth R/I^j+\depth S/J^{s-j}\}\\
             & \le \depth R/I+\depth S/J =\depth T/P.
\end{align*}
Since $P=I+J$ has a constant depth function, we conclude that $\depth I^s=\depth I$ and $\depth J^s=\depth J$ for all $s\ge 1$. The proof is finished.
\end{proof}
An upper bound for the index of depth stability of $P$ in terms of those of $I$ and $J$ was not known, according to \cite[page 821, paragraph after Theorem 5.6]{HTT}. With the assumptions of Theorem \ref{thm_depthreg}, we are able to provide such a bound. Theorems \ref{thm_asymptoticproj_polynomial} and \ref{thm_asymptoticreg} yield the following inequalities. The second of these inequalities also provides a slightly better upper bound for $\rstab(I+J)$ than \cite[Corollary 5.8]{HTT}.
\begin{cor}
Assume that either $\chara k=0$ or $I$ and $J$ are monomial ideals. Then $\dstab(I+J) \le \dstab(I)+\dstab(J)$.
If furthermore, in the notations of Theorem \ref{thm_asymptoticreg}, we have $e=f$, then $\rstab(I+J) \le \rstab(I)+\rstab(J)$.
\end{cor}
The following statement is possibly of independent interest.
\begin{cor}
\label{cor_depthregpowers}
Assume that either $\chara k=0$ or $I$ is a monomial ideal. Then for all $s\ge 1$, the map $\Tor^R_i(k,I^s) \to \Tor^R_i(k,I^{s-1})$ is zero for every $i\ge 0$. Moreover, there are equalities
\begin{align*}
\depth(I^{s-1}/I^s)&=\min\{\depth I^{s-1}, \depth I^s-1\},\\
\reg (I^{s-1}/I^s)&=\max\{\reg I^{s-1}, \reg I^s-1\}.
\end{align*}
\end{cor}
\begin{proof}
The first statement follows from combining Lemma \ref{lem_mapsofTor}(ii) and Theorem \ref{thm_monomialidealsareofsmalltype}. The equalities follow from Lemma \ref{lem_Tor-vanishingandinvariants} and the Auslander--Buchsbaum formula.
\end{proof}
\begin{rem}
From Corollary \ref{cor_depthregpowers}, we deduce that the stability index of the function $\depth I^{s-1}/I^s$ is  at most the stability index of the function $\depth R/I^s$, at least if $\chara k=0$ or $I$ is monomial. This result is quite surprising to us, as it was not so clear that {\it a priori}, there is such a relation without any assumption on the characteristic (see \cite[page 832]{HTT}).
\end{rem}
%-------------------------------------------------------
%-------------------------------------------------------
%-------------------------------------------------------
%-------------------------------------------------------
%-------------------------------------------------------
%-------------------------------------------------------

\section{Linearity defect}
\label{sect_lindofpowers}
In this section, we characterize the (asymptotic) linearity defect of powers of $P$ in terms of the data of $I$ and $J$. The goal is to find formulas similar to those of Proposition \ref{prop_formula_depthreg} and Theorem \ref{thm_asymptoticproj_general}. If $I$ and $J$ are of small type, using Betti splittings, we get similar upper bounds which in many situations are actual values for $\lind_T P^s$: In fact the equality occurs if $I$ and $J$ are of doubly small type. There is a lower bound $\max\{\lind_R I^s, \lind_S J^s\} \le \lind_T P^s$ following from Lemma \ref{lem_retract}, but this is far from sharp even for the starting case $s=1$, as seen from Proposition \ref{prop_invariants:mixedsum}(ii).

\subsection{Upper bounds and formulas}
The first main result of this section is the following. Part (ii) of the next result is the motive for introducing  ideals of doubly small type.
\begin{thm}
\label{thm_boundld_mixedsum}
Let $I$ and $J$ be ideals of small type. Assume that $I^i\neq 0, J^i\neq 0$ for all $i\ge 1$, e.g. $R$ and $S$ are reduced. Then the following statements hold:
\begin{enumerate}[\quad \rm(i)]
\item For all $s\ge 1$, we have an inequality
\[
\lind_T P^s \le \max_{i\in [1,s-1], j\in [1,s]}\left\{\lind_R I^{s-i}+\lind_S J^i,\lind_R I^{s-j+1}+\lind_S J^j+1\right\}.
\]
\item Assume further that both $I$ and $J$ are of doubly small type. Then the inequality in \textup{(i)} is an equality for all $s\ge 1$.
\end{enumerate}
\end{thm}
\begin{proof}
(i) By Theorem \ref{thm_sumsofsmalldoublysmall}, the decomposition $I^rP^s=I^{r+1}P^{s-1}+I^rJ^s$ is a Betti splitting. Hence by Proposition \ref{prop_Bettisplittings} and Lemma \ref{lem_intersect_quotient}, we obtain
\begin{equation}
\label{eq_ineq_ldIrPs_short}
\lind_T (I^rP^s) \le \max\{\lind_T (I^{r+1}P^{s-1}), \lind_T (I^rJ^s), \lind_T (I^{r+1}J^s)+1\}.
\end{equation}
Using Lemma \ref{lem_tensor}, we then get
\[
\lind_T (I^rP^s) \le \max\{\lind_T (I^{r+1}P^{s-1}), \lind_R I^r+\lind_S J^s, \lind_R I^{r+1}+\lind_S J^s+1\}.
\]
Using induction on $s$, we infer that
\begin{equation}
\label{eq_ineqldIrPs}
\lind_T (I^rP^s) \le \max_{i\in [0,s], j\in [1,s]}\left\{\lind_R I^{r+s-i}+\lind_S J^i,\lind_R I^{r+s-j+1}+\lind_S J^j+1\right\}.
\end{equation}
Setting $r=0$ in (i), the conclusion is
\[
\lind_T P^s \le \max_{i\in [0,s], j\in [1,s]}\left\{\lind_R I^{s-i}+\lind_S J^i,\lind_R I^{s-j+1}+\lind_S J^j+1\right\}.
\]
The term on the right-hand side corresponding to $i=0$ can be omitted, since it is smaller than the term corresponding to $j=1$. Similarly, the term on the right-hand side corresponding to $i=s$ can be omitted, since it is smaller than the term corresponding to $j=s$. Hence the desired inequality follows.

(ii) Since $I$ and $J$ are of doubly small type, by Theorem \ref{thm_sumsofsmalldoublysmall}, the maps $I^{r+1}J^s\to I^rJ^s$ and $I^{r+1}J^s\to I^{r+1}P^{s-1}$ are doubly Tor-vanishing. Consider the exact sequence $
0 \longrightarrow I^{r+1}J^s \longrightarrow I^rJ^s\oplus I^{r+1}P^{s-1} \longrightarrow I^rP^s \longrightarrow 0$.
As all the powers of $I$ and $J$ are non-trivial, we also have $I^iJ^j\cong I^i\otimes_k J^j \neq 0$ for all $i,j\ge 0$. Applying Lemma \ref{lem_Tor-vanishingandinvariants}, we see that \eqref{eq_ineq_ldIrPs_short} is now an equality. Hence so is \eqref{eq_ineqldIrPs}, giving the desired conclusion. 
\end{proof}
\begin{ex}
\label{ex_notofsmalltype_ld}
The following example shows that the conclusion of Theorem \ref{thm_boundld_mixedsum} may fail if one of the ideals $I$ and $J$ is not of small type. The example is quite similar to Example \ref{ex_nonpolynomialbase}, and it shows that even if $R$ is regular and $S$ is a non-regular Koszul ring, powers of the mixed sum $I+J$ may have complicated behavior. 

Consider the rings and ideals $R=k[a,b,c], I=(a^4,a^3b,ab^3,b^4,a^2b^2c^5)$ and $S=k[x,y,z]/(x^2,yz), J=(x+y)$. We will show that
\begin{enumerate}
\item $I$ is of small type, $J$ is neither of small type nor nilpotent,
\item $\lind_R I=\lind_S J=1, \lind_R I^2=\lind_S J^2=0$,
\item $\lind_T P^2\ge 3 > 2= \max\{\lind_R I+\lind_S J, \lind_R I^2+\lind_S J+1, \lind_R I+\lind_S J^2+1\}$.
\end{enumerate}

(i) We have $J^n=(x+y)y^{n-1}$ for $n\ge 2$, hence $J$ is not nilpotent. By direct inspection, $0:(x+y)=(xz)$, hence $\beta_{1,3}(J)\neq 0$. On the other hand, $(x+y)^2:(x+y)=(x+y)$, hence $\beta_{1,3}(J/J^2)=0$. The map $\Tor^S_1(k,J) \to \Tor^S_1(k,J/J^2)$ being not injective, so $J$ is not of small type. We know by Theorem \ref{thm_monomialidealsareofsmalltype} that $I$ is of small type.

(ii) Since $S$ is Koszul and $\reg_S J\ge 2$, $\lind_S J\ge 1$. On the other hand, $0:(x+y)=(xz)$ and as $S$ is defined a by quadratic monomial ideal, $0:(xz)=(x,y)$ has a $1$-linear resolution over $S$. This implies that $\lind_S J=1$.

Observe that $J^2=(xy+y^2)$ and $0:(xy+y^2)=(z)$ has a $1$-linear resolution over $S$, so $\lind_S J^2=0$. By direct computation with Macaulay2, $\lind_R I=1$. Clearly $\lind_R I^2=0$ since $I^2=(a,b)^8$.

(iii) By computations with Macaulay2 \cite{GS}, $\beta_{3,10}(P^2)\neq 0$ while $\beta_{2,9}(P^2) = 0$. Let $F$ be the minimal graded free resolution of $P^2$ over $T$. By the minimality of $F$,  we see that in $\linp^T F$, the map $(\linp^T F)_3 \to (\linp^T F)_2$ has non-trivial kernel. This implies that $\lind_T P^2\ge 3$, as desired.
\end{ex}
The second main result of this section is an analog of Proposition \ref{prop_formula_projreg_specialcases}(iii). We were benefited by the method of \cite{HTT} in Step 2 of the proof.
\begin{thm}
\label{thm_sumswithlinearKoszulpowers}
Assume that $J$ is of small type and  $I$ is an ideal generated by linear forms of $R$ such that all the powers of $I$ are Koszul. Assume further that $I^i, J^i\neq 0$ for all $i\ge 1$. Then for all $s\ge 1$, we have $\lind_T P^s = \max_{i\in [1,s]}\{\lind_S J^i\}$.
\end{thm}
\begin{proof}
We prove more generally for all $r\ge 0, s\ge 1$ the equality $\lind_T (I^rP^s) =$ $\max_{i\in [1,s]} \{\lind_S J^i\}$.

\textsf{Step 1}: Consider the following exact sequence
\[
0 \longrightarrow I^{r+1}P^{s-1} \longrightarrow I^rP^s \longrightarrow I^rJ^s/(I^{r+1}J^s)\cong (I^r/I^{r+1})\otimes_k J^s \longrightarrow 0.
\]
Since $J$ is of small type, by the claim in the proof of Proposition \ref{prop_formula_projreg_specialcases}, the map $\Tor^R_{i+1}(k, I^rJ^s/(I^{r+1}J^s)) \to$ $ \Tor^R_i(k,I^{r+1}P^{s-1})$ is zero for all $i\ge 0$. By \cite[Proposition 2.5(ii)]{Ng1}, there is an inequality
\begin{equation}
\label{eq_ineq_lindJrIs_linearJ}
\lind_T (I^rP^s) \le \max\left\{\lind_T (I^{r+1}P^{s-1}), \lind_T \left((I^r/I^{r+1})\otimes_k J^s\right)\right\}.
\end{equation}
Using Theorem \ref{thm_m-smallext}, we prove that $I^r/I^{r+1}$ is a Koszul $R$-module. For this, we need to check that (a) $I^r$ and $I^{r+1}$ are Koszul, and (b) $I^{r+1}\cap \mm^{s+1}I^r=\mm^sI^{r+1}$ for all $s\ge 0$.
The hypothesis guarantees (a), while (b) holds true by degree reason and the fact that $I$ is generated by linear forms. Hence by Theorem \ref{thm_m-smallext}, $I^r/I^{r+1}$ is Koszul. 

Together with Lemma \ref{lem_tensor} and \eqref{eq_ineq_lindJrIs_linearJ}, we get the inequality 
$$
\lind_T (I^rP^s) \le \max \{\lind_T (I^{r+1}P^{s-1}), \lind_S J^s\}.
$$
Using induction on $s\ge 0$, we get $\lind_T (I^rP^s) \le \max_{i\in [1,s]} \{\lind_S J^i\}.$  

\textsf{Step 2}: It remains to show that $\lind_T (I^rP^s) \ge \max_{i\in [1,s]} \{\lind_S J^i\}$
for all $r\ge 0, s\ge 1$. There is nothing to do if the right-hand side is $0$, hence we will assume its positivity.

By Lemmas \ref{lem_intersect_quotient} and \ref{lem_tensor}, there is an equality
\[
\lind_T \left(I^rP^{s-1}/I^rP^s\right)=\max_{i\in [0,s-1]}\left\{\lind_R (I^{r+i}/I^{r+i+1})+\lind_S (J^{s-i-1}/J^{s-i})\right\}.
\]
As proved in Step 1, $I^r/I^{r+1}$ is Koszul for all $r$. So we get
\[
\lind_T \left(I^rP^{s-1}/I^rP^s\right)=\max_{i\in [1,s]}\left\{\lind_S (J^{i-1}/J^i)\right\}.
\]
\textbf{Claim:} Given $\max_{i\in [1,s]} \{\lind_S J^i\}\ge 1$, it holds that
\begin{equation}
\label{eq_comparisionld}
\max_{i\in [1,s]}\left\{\lind_S (J^{i-1}/J^i)\right\}=\max_{i\in [1,s]} \{\lind_S J^i\}+1.
\end{equation}
{\it Proof of the claim.} Since $J$ is of small type, Proposition \ref{prop_zeroinducedmapofTor} yields $\lind_S (J^{i-1}/J^i) \le \max\{\lind_S J^{i-1}, \lind_S J^i+1\}$ for $1\le i\le s$. These inequalities imply that the left-hand side of \eqref{eq_comparisionld} is $\le$ the right-hand side.

For the reverse inequality, again using Proposition \ref{prop_zeroinducedmapofTor}, we have 
$$
\lind_S J^i \le \max\{\lind_S J^{i-1}, \lind_S (J^{i-1}/J^i)-1\}
$$ 
for $2\le i\le s$. Hence it suffices to show that $\lind_S J+1 \le$ the left-hand side of \eqref{eq_comparisionld}. If $\lind_S (S/J)\ge 1$, we have $\lind_S J+1=\lind_S (S/J)$ and we are done. If $\lind_S (S/J)=0$ then also $\lind_S J=0$ and the desired inequality holds unless the left-hand side of \eqref{eq_comparisionld} is zero. But if that happens then by induction on $i$ and Proposition \ref{prop_zeroinducedmapofTor}, we have $\lind_S J^i=0$ for $1\le i\le s$. This contradicts the assumption $\max_{i\in [1,s]}\{\lind_S J^i\}\ge 1$. So $\lind_S J+1 \le$ the left-hand side of \eqref{eq_comparisionld} and the claim follows.

The previous discussions show that
\[
\lind_T \left(I^rP^{s-1}/I^rP^s\right)=\max_{i\in [1,s]} \{\lind_S J^i\}+1.
\]
By the proof of Theorem \ref{thm_sumsofsmalldoublysmall}, the map $I^rP^s\to I^rP^{s-1}$ is Tor-vanishing. Hence by Proposition \ref{prop_zeroinducedmapofTor}, there is an inequality
\[
\lind_T \left(I^rP^{s-1}/I^rP^s\right) \le \max\{\lind_T (I^rP^{s-1}),\lind_T (I^rP^s)+1\}.
\]
Putting everything together, we obtain
\begin{equation*}
\max_{i\in [1,s]}\{\lind_S J^i\}+1 \le \max\{\lind_T (I^rP^{s-1}),\lind_T (I^rP^s)+1\}.
\end{equation*}
Recall that from Step 1, $\lind_T (I^rP^{s-1}) \le \max_{i\in [1,s-1]}\{\lind_S J^i\}$, hence $\lind_T (I^rP^s) \ge \max_{i\in [1,s]}\{\lind_S J^i\}$, as desired. The proof is concluded.
\end{proof}
\begin{ex}
The conclusion of Theorem \ref{thm_sumswithlinearKoszulpowers} does not hold if $J$ is not of small type. Consider $R=k[a], I=(a), S=k[x,y,z]/(x^2,yz)$ and $J=(x+y)$. By Example \ref{ex_notofsmalltype_ld}, $J$ is not of small type. We claim that in $T=R\otimes_k S$, $\lind_T P^2 \ge 2 > 1=\max\{\lind_S J, \lind_S J^2\}$. Indeed, by computations with Macaulay2 \cite{GS}, $\beta_{2,5}(P^2)\neq 0$ while $\beta_{1,j}(P^2)=0$ for $j\ge 4$. Similarly to Example \ref{ex_notofsmalltype_ld}, $\lind_T P^2 \ge 2$.
\end{ex}
\subsection{Asymptotes}
Denote by 
$$
\glind R=\sup\{\lind_R M:  \text{$M$ is a finitely generated graded $R$-module}\}
$$ 
the so-called {\it global linearity defect} of $R$. For example, if $R$ is regular (or more generally a quadric hypersurface) then $\glind R= \dim R$; see \cite[Corollary 6.4]{CINR}. Let the {\it index of linearity defect stability} of $I$, $\lstab(I)$, to be the least positive integer $r$ such that for all $i\ge r$, $\lind_R I^i$ is a constant depending only on $I$. By \cite[Theorem 1.1]{NgV1b}, if $\glind R<\infty$ then $\lstab(J)$ is well-defined and finite for every $I$. 

The third main result of Section \ref{sect_lindofpowers} is
\begin{thm}
\label{thm_asymptoticlind}
Assume that $R$ and $S$ have finite global linearity defect and $I^i\neq 0, J^i\neq 0$ for all $i\ge 1$. Assume further that $I$ and $J$ are of doubly small type. Then for all $s\ge \lstab(I)+\lstab(J)$, we have
\[
\lind_T P^s=\max\left\{\lim_{i\to \infty}\lind_R I^i+\max_{j \ge 1}\lind_S J^j+1 , \max_{i\ge 1}\lind_R I^i+\lim_{j\to \infty}\lind_S J^j+1\right\}.
\]
\end{thm}
\begin{proof}
The proof is parallel to that of Theorem \ref{thm_asymptoticproj_general}, taking Theorem \ref{thm_boundld_mixedsum}(ii) into account. 
\end{proof}
\begin{cor}
With the notations and assumptions of Theorem \ref{thm_asymptoticlind}, there is an inequality $\lstab(I+J)\le \lstab(I)+\lstab(J)$.
\end{cor}
We are able to compute $\lind_T P^s$, $s$ large enough, in some fairly general cases.
\begin{cor}
\label{cor_asymptote_specialcases}
Let $(R,\mm), (S,\nn)$ be polynomial rings over $k$. Let $(0) \neq I \subseteq \mm^2, (0)\neq J \subseteq \nn^2$ be homogeneous ideals of $R,S$, resp. Assume that one of the following conditions holds:
\begin{enumerate}[\quad \rm(i)]
\item $\chara k=0$ and $I=U^p$ and $J=V^q$ where $p,q\ge 2$ and $U,V$ are homogeneous ideals of $R,S$, resp.
\item $I=U^p$ and $J=V^q$ where $p,q \ge 2$ and $U,V$ are monomial ideals of $R,S$, resp.
\item All of the powers of $I$ and $J$ are Koszul.
\end{enumerate}
Then for all $s\ge \lstab(I)+\lstab(J)$, there is an equality
\[
\lind_T P^s=\max\left\{\lim_{i\to \infty}\lind_R I^i+\max_{j \ge 1}\lind_S J^j+1 , \max_{i \ge 1}\lind_R I^i+\lim_{j\to \infty}\lind_S J^j+1\right\}.
\]
\end{cor}
\begin{proof}
This follows by combining Propositions \ref{prop_idealswithKoszulpowers}, \ref{prop_stronglyGolodisofdoublysmalltype} and Theorem \ref{thm_asymptoticlind}.
\end{proof}
If additionally, one of $I$ and $J$ is generated by linear forms, we have a nicer asymptotic statement.
\begin{thm}
\label{thm_asymptoticlind_linear}
Let $R, S$ be polynomial rings over $k$. Let $I$ be generated by linear forms. Assume further that one of the following conditions holds:
\begin{enumerate}[\quad \rm(i)]
\item $\chara k=0$.
\item $J$ is a monomial ideal of $S$.
\item All the powers of $J$ are Koszul.
\end{enumerate}
Then for all $s\ge 1$, there is an equality $\lind_T P^s=\max_{i\in [1,s]}\{\lind_S J^i\}$.
Moreover, if $1\le p\le \lstab(J)$ be minimal such that $\lind_S J^p=\max_{i\ge 1} \{\lind_S J^i\}$ then $\lind_T P^s=\max_{i\ge 1}\{\lind_S J^i\}$ for all $s\ge p$ and $\lstab(I+J)=p$.
\end{thm}
\begin{proof}
By Lemma \ref{lem_mapsofTor}(ii), Theorem \ref{thm_monomialidealsareofsmalltype} and Proposition \ref{prop_idealswithKoszulpowers}, $J$ is of small type. As all the powers of $I$ have a linear resolution, the result is immediate from Theorem \ref{thm_sumswithlinearKoszulpowers}. 
\end{proof}
\subsection{Ideals with Koszul powers}
\begin{cor}	
Assume that $R$ and $S$ are reduced and $I$ is generated by linear forms. The following statements are equivalent:
\begin{enumerate}[\quad \rm(i)]
\item All the powers of $I$ and $J$ are Koszul ideals;
\item All the powers of $P$ are Koszul ideals.
\end{enumerate}
\end{cor}

\begin{proof}
For (ii) $\Longrightarrow$ (i): Take $s\ge 1$. Since $\lind_T P^s=0$, using Lemma \ref{lem_retract}, we get $\lind_R I^s \le \lind_T P^s=0$. Hence $I$ has Koszul powers. Similar arguments work for $J$.

For (i) $\Longrightarrow$ (ii): Applying Proposition \ref{prop_idealswithKoszulpowers}, we see that $J$ is of small type. By Theorem \ref{thm_sumswithlinearKoszulpowers}, we conclude the proof.
\end{proof}
The following result supplies computation of the linearity defect of powers for a non-trivial class of ideals.
\begin{thm}
Let $(R_1,\mm_1),\ldots,(R_c,\mm_c)$ be Koszul, reduced $k$-algebras, where $c\ge 1$. For each $i=1,\ldots,c$, let $(0) \neq I_i \subseteq \mm_i^2$ be a homogeneous ideal of $R_i$ such that all the powers of $I_i$ are Koszul. Denote $P=I_1+I_2+\cdots+I_c\subseteq T=R_1\otimes_k \cdots \otimes_k R_c$ the mixed sum of $I_1,\ldots,I_c$. Then $\lind_T P^s=c-1$ for all $s\ge 1$. 
\end{thm}
\begin{proof}
We proceed by induction on $c\ge 1$. The case $c=1$ is a tautology. Assume that $c\ge 2$.

By the hypothesis and Proposition \ref{prop_idealswithKoszulpowers}, $I_i$ is of doubly small type for all $1\le i\le c$. Denote $J=I_2+\cdots+I_c \subseteq S=R_2\otimes_k \cdots \otimes_k R_c$, then by Theorem \ref{thm_sumsofsmalldoublysmall}(ii), $J$ is of doubly small type. By the induction hypothesis, $\lind_S J^s=c-2$ for all $s\ge 1$. Note that all the powers of $J$ are non-zero, since all the powers of $I_2,\ldots,I_c$ are so. Hence applying Theorem \ref{thm_boundld_mixedsum}(ii) to the sum $P=I_1+J$, we get that for any $s\ge 1$,
\[
\lind_T P^s=\max_{i\in [1,s-1], j\in [1,s]}\{\lind_{R_1} I_1^{s-i}+\lind_S J^i, \lind_{R_1} I_1^{s-j+1}+\lind_S J^j+1 \}=c-1.
\]
This concludes the induction and the proof.
\end{proof}

\subsection*{Acknowledgments}
This project was completed in part thanks to our visits to the University of Nebraska -- Lincoln in March 2015 and the Vietnam Institute for Advanced Study in Mathematics (VIASM) in April and May 2016. We are grateful to the highly efficient staff and the inspiring working environment at the VIASM, which made our stay an enjoyable experience.  We are grateful to N.V. Trung, T.N. Trung for many useful suggestions related to this paper. Thanks are also due to Luchezar Avramov, Hailong Dao and Jeff Mermin for some useful discussions on the terminology. Both authors thank T. H\`a-N.V. Trung-T.N. Trung and J. Herzog-C.Huneke for the friendly discussions about their papers \cite{HTT} and \cite{HeHu}, respectively.

We would like to thank the anonymous referee not only for many valuable suggestions on professionalism, but also for useful comments on the presentation of our paper.

The first author was a Marie Curie fellow of the Istituto Nazionale di Alta Matematica.
%-------------------------------------------------------
%-------------------------------------------------------
%-------------------------------------------------------
%-------------------------------------------------------
%-------------------------------------------------------
%-------------------------------------------------------

\end{document}